\documentclass[11pt]{article}

\usepackage[a4paper,margin=1in]{geometry}
\usepackage{amsmath,amssymb,amsthm,mathtools}
\usepackage{booktabs,tabularx}
\usepackage{enumitem}
\usepackage{xcolor}
\usepackage{graphicx}
\usepackage{microtype}
\usepackage[colorlinks=true,linkcolor=blue!55!black,citecolor=blue!55!black,urlcolor=blue!55!black]{hyperref}
\usepackage[nameinlink,noabbrev]{cleveref}
\usepackage{authblk}

\allowdisplaybreaks[3]
\setlist[itemize]{leftmargin=1.7em,itemsep=0.2em,topsep=0.3em}

\newtheorem{theorem}{Theorem}[section]
\newtheorem{lemma}[theorem]{Lemma}
\newtheorem{proposition}[theorem]{Proposition}
\newtheorem{corollary}[theorem]{Corollary}

\theoremstyle{definition}
\newtheorem{remark}[theorem]{Remark}

\newcommand{\R}{\mathbb{R}}
\newcommand{\E}{\mathbb{E}}
\newcommand{\Prob}{\mathbb{P}}
\newcommand{\ip}[2]{\big\langle #1,#2\big\rangle}

\newcommand{\rank}{\operatorname{rank}}
\newcommand{\diag}{\operatorname{diag}}
\newcommand{\cA}{\mathcal{A}}
\newcommand{\err}{\mathcal{E}}
\newcommand{\eps}{\varepsilon}
\newcommand{\dist}{\mathfrak D}

\title{Convergence Rates of Factorized Gradient Descent for\\
Overparameterized Asymmetric Matrix Sensing with Balancing Regularization%
\thanks{Liping Zhang is supported by the National Natural Science Foundation of China under Grant No.~12571323.}}

\author[1]{Yongjun Chen}
\author[1]{Jiaming Zhang}
\author[1]{Liping Zhang\thanks{Corresponding author.
E-mail: lipingzhang@mail.tsinghua.edu.cn}}

\affil[1]{%
Department of Mathematical Sciences, Tsinghua University,\\
Beijing 100084, China\\
E-mail: yongjun-24@mails.tsinghua.edu.cn;\
jiaming-24@mails.tsinghua.edu.cn;\
lipingzhang@mail.tsinghua.edu.cn
}

\date{}

\begin{document}
\maketitle

\begin{abstract}
We study factorized gradient descent (FGD) for overparameterized asymmetric matrix sensing with the balancing regularizer $\frac18\|L^\top L-R^\top R\|_F^2$. An exact indefinite lifting of the population objective decomposes the dynamics of FGD into positive, negative, and null spectral subspaces. The first two components have linear restoring terms, whereas the null component contracts only quadratically. For independent isotropic sub-Gaussian sensing matrices and independent centered sub-Gaussian noise with scale proxy $\sigma$, we prove the local bound
\[
 \|L_tR_t^\top-X^\star\|_2
 \lesssim\max\left\{\frac{1}{\eta t+1/\sigma_r},
 \kappa\sigma\sqrt{\frac{d_1+d_2}{n}}\right\},
\]
under a sufficiently small initialization error,
the number of linear measurements $n\gtrsim\kappa^2(d_1+d_2)\max\{k^2,\sigma^2/\sigma_r^2\}$, and step size $\eta\lesssim\sigma_1^{-1}$. Here $\sigma_1$ and $\sigma_r$ are the largest and smallest nonzero singular values of the target $X^\star\in \R^{d_1\times d_2}$ respectively, and $\kappa=\sigma_1/\sigma_r$. Furthermore, a population trajectory was constructed to have tight product error $\Theta(1/t)$. These results provide a local rate characterization for ordinary gradient descent for overparameterized asymmetric matrix sensing with balancing regularizer.
\end{abstract}


\section{Introduction}
Low-rank matrix sensing seeks to recover an unknown rank-\(r\) matrix
\(X^\star\in\R^{d_1\times d_2}\) from noisy linear measurements
\begin{equation}
y_i=\langle A_i,X^\star\rangle+\epsilon_i,
\quad i=1,\ldots,n.
\label{eq:intro-model}
\end{equation}
This model arises in signal processing, array data analysis, image reconstruction, and quantum-state tomography
\cite{chi2019,dokmanic2015edm,gross2010quantum,wen2012,zhao2010}.
Its factorized formulation also provides a useful testbed for studying optimization dynamics and convergence behavior since it is nonconvex and retains several characteristic phenomena \cite{arora2019deep,gunasekar2017,li2018,li2021greedy}.

In high-dimensional settings, a standard approach is to represent the unknown matrix as the product of two low-dimensional factors,
$X=LR^\top, L\in\R^{d_1\times k},  R\in\R^{d_2\times k}$,
thereby avoiding repeated large-scale singular value decompositions during optimization. Gradient descent is a standard method for solving the resulting factorized problem. With a correctly specified factor width \(k=r\), nonconvex analyses establish local linear convergence under suitable measurement and initialization conditions, together with favorable landscape properties of the factored objective \cite{tu2016,chenwainwright2015,ge2017,park2018}; see \cite{chi2019} for an overview. 

In practice, however, the target rank \(r\) is usually unknown, and one often chooses a factor width \(k\ge r\). It is called overparameterized when \(k>r\), and overparameterization plays a key role in optimization, especially for
neural networks  \cite{jacot2018neural,zou2020gradient,fang2021modeling}. The overspecification increases modeling flexibility but also introduces redundant directions with degenerate local curvature near a low-rank solution, making the landscape riddled with saddle points and lacks global smoothness and a global PL condition. %This problem has thus become a canonical benchmark in theoretical deep learning \cite{wei2026benefits}.

%For overparameterized symmetric matrix sensing, Zhuo et al. \cite{zhuo2024} characterized the dynamics and statistical accuracy of ordinary gradient descent under noisy measurements. In the noiseless asymmetric setting, Xiong et al. \cite{xiong2024} established initialization-dependent geometric convergence under their random-initialization and sensing assumptions. 

Asymmetric factorizations are non-identifiable since $LR^\top=(LH)(RH^{-\top})^\top$ for every invertible matrix \(H\). This invariance can create a severe scale imbalance between the two factors. A standard way to control this ambiguity is to augment the objective with the balancing regularizer $\|L^\top L-R^\top R\|_F^2$. This penalty is widely used in error analyses and landscape theory for asymmetric matrix sensing \cite{ge2017,park2018,zhu_geometry2021}. Although the regularized objective and its population geometry have been studied, the precise local convergence rate of gradient descent in the noisy overparameterized setting remains less well understood. 

Motivated by this gap, we analyze the discrete-time dynamics of ordinary gradient descent for overparameterized asymmetric matrix sensing with an explicit balancing regularizer.

\subsection{The regularized model}

Let $y_i=\langle A_i,X^\star\rangle+\epsilon_i$, where
$X^\star\in\R^{d_1\times d_2}$ has rank $r$. For a chosen width $k\ge r$, consider
\begin{equation}
 \min_{L\in\R^{d_1\times k},\,R\in\R^{d_2\times k}}
 \Phi_n(L,R)
 :=\frac1{2n}\sum_{i=1}^n
 (\langle A_i,LR^\top\rangle-y_i)^2
 +\frac18\|L^\top L-R^\top R\|_F^2.
 \label{eq:intro-objective}
\end{equation}
We analyze simultaneous gradient updates of $L$ and $R$ with a common constant step size. Our central question is
\begin{quote}
What local convergence rate and noise accuracy does gradient descent achieve for explicitly balanced asymmetric matrix sensing when the factor width is overspecified?
\end{quote}

\subsection{Results and technical contributions}

Write $D=d_1+d_2$ and $\kappa=\sigma_1(X^\star)/\sigma_r(X^\star)$.
Our main result is a local recovery guarantee for gradient descent with the balancing coefficient fixed at $1/8$. This coefficient makes the problem particularly tractable. With $Z=[L^\top,R^\top]^\top$ and $C^\star=\left[\begin{smallmatrix}0&X^\star\\X^{\star\top}&0\end{smallmatrix}\right]$, the population objective equals $\frac18\|ZZ^\top-2C^\star\|_F^2$ up to a constant. With this indefinite lifting, we decompose the spectral space of the population objective into three parts, which shows the relation to symmetric analyses while exposing the additional blocks needed for the asymmetric regularized dynamics. 

\begin{enumerate}
\item \textbf{Local recovery rates under noise.}
For independent isotropic sub-Gaussian sensing matrices and independent centered sub-Gaussian noise with scale proxy $\sigma$, we prove with high probability that
\[
 \|L_tR_t^\top-X^\star\|_2
 \lesssim\max\left\{\frac1{\eta t+1/\sigma_r(X^\star)},
 \kappa\sigma\sqrt{\frac Dn}\right\},\quad t\ge0.
\]
The guarantee assumes a sufficiently close initialization,
$\eta\lesssim1/\sigma_1(X^\star)$, and
\[
n\gtrsim\kappa^2D\max\{k^2,\sigma^2/\sigma_r(X^\star)^2\}.
\]

\item \textbf{Spectral mechanisms governing the local rate.}
We derive exact discrete-time recursions for six Gram blocks associated with the positive, negative, and null spectral subspaces of the population lifting. Blocks coupled to the first two spectral subspaces have linear restoring terms, whereas the block in the null spectral subspace contracts quadratically. 

\item \textbf{A tight sublinear example.}
A structured population trajectory has tight product error $\Theta(1/t)$. The empirical null-block estimate retains first-order forcing proportional to the current block error, which yields an $O(\kappa\sigma\sqrt{D/n})$ recovery radius and squared Frobenius error $O(k\kappa^2D\sigma^2/n)$ once the optimization term reaches that scale. 
\end{enumerate}

\subsection{Related work}
\label{sec:comparison}

Convex approaches, such as nuclear-norm minimization, provide fundamental statistical recovery guarantees for matrix sensing
\cite{candesplan2011,recht2010}. Their computational cost, however, can become prohibitive in high-dimensional settings. Factorized methods instead represent the unknown matrix as the product of two low-dimensional factors, thereby reducing both the number of variables and the computational cost. This parameterization introduces nonconvexity and poses additional challenges for theoretical analysis. Gradient descent and its variants form a standard class of algorithms for solving the resulting factorized problem.

In the exactly parameterized setting \(k=r\), representative works by Zheng and Lafferty \cite{zheng2015} and Tu et al. \cite{tu2016} establish linear convergence of gradient-based methods under suitable measurement and initialization conditions. For noiseless overparameterized symmetric matrix sensing, where \(k>r\), St{\"o}ger and Soltanolkotabi \cite{stoger2021} show that gradient descent with a sufficiently small random initialization achieves recovery guarantees that depend on the initialization scale. Zhuo et al. \cite{zhuo2024} study the noisy overparameterized symmetric setting and show that ordinary gradient descent converges sublinearly to the statistical error. Jin et al. \cite{jin2023} provide a fine-grained analysis of the corresponding dynamics and show that gradient descent with small initialization learns the spectral components incrementally.  For noiseless asymmetric matrix sensing, Xiong et al. \cite{xiong2024} establish initialization-dependent geometric convergence under their random-initialization and sensing assumptions, while Wind \cite{wind2025} studied the robustness of the dynamics to measurement noise and stochastic perturbations.

A separate line of work modifies the optimization metric to mitigate the effects of overparameterization and ill-conditioning. Scaled and preconditioned gradient methods have been developed for low-rank matrix estimation and overparameterized matrix sensing
\cite{tong2021,zhang2021,xu2023,zhang2024noisy}. In particular, Zhang et al. \cite{zhang2021} introduce a preconditioned gradient method, showing local linear convergence for overparameterized symmetric matrix sensing. For noisy overparameterized asymmetric matrix sensing, alternating preconditioned gradient descent has also been shown to converge rapidly to the corresponding statistical accuracy \cite{liu2025apgd}.

The works above primarily analyze the trajectories generated by gradient decent algorithm. A complementary line of research studies the optimization landscape of factored low-rank problems. Under suitable restricted isometry or regularity conditions, these results establish benign geometric properties, such as the absence of spurious local minima and the existence of negative curvature at non-global critical points
\cite{ge2017,zhu2018,li2019,zhu_geometry2021}. %Zhu et al. \cite{zhu2018} analyze general low-rank matrix optimization and establish global optimality and strict-saddle properties under appropriate restricted strong convexity and smoothness conditions. Li et al. \cite{li2019} develop a general framework for characterizing the effects of symmetry on the stationary points and global landscape of nonconvex matrix factorization. Zhu et al. \cite{zhu_geometry2021} further sharpen the global geometric analysis for both symmetric and asymmetric low-rank optimization.
More generally, Zhang et al. \cite{zhang2021a} establish explicit restricted-isometry conditions under which general low-rank minimization problems have no spurious second-order critical points and satisfy the strict-saddle property. Bi et al. \cite{bi2022} verify a local Polyak-{\L}ojasiewicz inequality, and obtain global convergence guarantees for perturbed gradient descent with a strict-saddle analysis. These landscape results are formulated primarily for the exactly parameterized setting \(k=r\); in the asymmetric case, the analysis commonly introduces a balancing regularizer of the form
$\|U^\top U-V^\top V\|_F^2.$

We review the literature on the convergence rate of ordinary gradient descent for matrix sensing in Table \ref{tab:gd-comparison}, and our paper fills the gap in the study of the convergence rate of gradient descent for noisy overparameterized asymmetric matrix sensing with an explicit balancing regularizer.

\begin{table}[t]
\centering
\small
\caption{Convergence results for gradient descent in matrix sensing. ``No'' indicates the absence of an explicit balancing penalty. }
\label{tab:gd-comparison}
\setlength{\tabcolsep}{4pt}
\renewcommand{\arraystretch}{1.15}
\begin{tabularx}{\textwidth}{@{}>{\raggedright\arraybackslash}p{2.8cm}>{\raggedright\arraybackslash}p{3.25cm}>{\raggedright\arraybackslash}p{2.85cm}>{\raggedright\arraybackslash}X@{}}
\toprule
Reference & Model/balancing & Rank/noise & Convergence rate \\
\midrule
Ma et al.\ \cite{mali2021} & asymmetric / no & Matched / yes & Local geometric recovery\\
Zhuo et al.\ \cite{zhuo2024} & symmetric / --- & Overspecified / yes & Local sublinear recovery\\
Xiong et al.\ \cite{xiong2024} & asymmetric / no &  Overspecified / no & Initialization-dependent geometric rate\\
% Zhang et al.\ \cite{zhang2024noisy} & PSD / --- & Preconditioned GD & Overspecified; local geometric recovery\\
% Liu et al.\ \cite{liu2025apgd} & asymmetric / no & Preconditioned GD & Overspecified; local geometric recovery\\
This paper & asymmetric / explicit & Overspecified/yes & Local sublinear recovery\\
\bottomrule
\end{tabularx}
\end{table}

\subsection{Organization}
\Cref{sec:main-result} states the model and main guarantees.
\Cref{sec:lifting} develops the lifted population dynamics and a tight example. \Cref{sec:concentration} proves uniform empirical control and the master recursion, and \cref{sec:main-proof} completes the convergence proofs. \Cref{sec:limitations} concludes. The appendices give
second-order expansions and the empirical gradient calculation.

\subsection{Notation}
For a positive integer $m$, write $[m]=\{1,\ldots,m\}$.  The transpose, rank, trace, and identity matrix are denoted by $M^\top$, $\rank(M)$, $\operatorname{tr}(M)$, and $I_m$, respectively; the dimension of an identity or zero matrix is omitted when it is clear.  The Frobenius inner product is $\ip{A}{B}=\operatorname{tr}(A^\top B)$.  We write $\|M\|_2$ and $\|M\|_F$ for the operator and Frobenius norms, and $\sigma_j(M)$ for the $j$-th largest singular value.  For symmetric matrices, $A\preceq B$ means that $B-A$ is positive semidefinite, while $\lambda_{\min}(A)$ and $\lambda_{\max}(A)$ denote the extreme eigenvalues.  The Orlicz norms $\|X\|_{\psi_2}$ and $\|X\|_{\psi_1}$ quantify sub-Gaussian and sub-exponential tails. The relation $a\lesssim b$
means $a$ is bounded above by $b$ up to a constant factor, i.e., $a\le Cb$, and the constant in this paper depending only on the fixed sub-Gaussian norm bounds $K_A$ and $K_\epsilon$ introduced below. If $a\lesssim b$ and $b \lesssim a$, then we say $a\asymp b$. All probability statements are with respect to the sensing matrices and observation noise.  Vacuous block rows and columns are interpreted as zero in boundary cases such as $d_1+d_2=2r$.

\subsection{Sub-Gaussian and sub-exponential variables}
For $\alpha\ge1$, let $\psi_\alpha(s)=\exp(s^\alpha)-1$ for $s\ge0$.  The
Luxemburg-Orlicz norm \cite{rao1991} generated by $\psi_\alpha$ is
\begin{equation}
 \|X\|_{\psi_\alpha}
 :=\inf\left\{t>0:
 \E\exp\!\left(\frac{|X|^\alpha}{t^\alpha}\right)\le2\right\}.
 \label{eq:orlicz-definition}
\end{equation}
We call $X$ \emph{sub-Gaussian} when $\|X\|_{\psi_2}<\infty$ and
\emph{sub-exponential} when $\|X\|_{\psi_1}<\infty$.  The number $2$ in
\eqref{eq:orlicz-definition} is conventional; replacing it by any fixed constant larger than one gives an equivalent norm. The probabilistic use is reviewed in \cite[Chapter~2]{vershynin2018}.  For $\alpha\in\{1,2\}$,
\begin{equation}
 \|X\|_{\psi_\alpha}
 \asymp \sup_{p\ge1}p^{-1/\alpha}(\E|X|^p)^{1/p},
 \quad
 \Prob(|X|\ge u)\le
 2\exp\!\left[-\left(\frac{u}{\|X\|_{\psi_\alpha}}\right)^\alpha\right].
 \label{eq:orlicz-characterizations}
\end{equation}
The comparison constants depend only on $\alpha$; the displayed tail bound
is used for positive Orlicz norm and $u>0$. 
A random vector $a\in\R^m$ is \emph{isotropic} if
$\E[aa^\top]=I_m$, and is $K$-sub-Gaussian if
\begin{equation}
 \sup_{x\in\mathbb S^{m-1}}
 \|\langle a,x\rangle\|_{\psi_2}\le K.
 \label{eq:subgaussian-vector}
\end{equation}

\section{Model and main results}
\label{sec:main-result}

This section formulates the noisy asymmetric matrix-sensing problem and presents the principal convergence guarantees studied in the remainder of the paper. 

\subsection{Observation model}

Let \begin{equation}
 X^\star=P\Sigma Q^\top,\quad
 \Sigma=\diag(\sigma_1,\ldots,\sigma_r),\quad
 \sigma_1\ge\cdots\ge\sigma_r>0
 \label{eq:svd}
\end{equation}
be the singular-value decomposition of a rank-$r$ target matrix in
$\R^{d_1\times d_2}$. Here $P\in\R^{d_1\times r}$ and
$Q\in\R^{d_2\times r}$ have orthonormal columns.  We write $D=d_1+d_2$ and $\kappa=\frac{\sigma_1}{\sigma_r}$ for the combined dimension and the condition number. We observe
\begin{equation}
 y_i=\ip{A_i}{X^\star}+\epsilon_i,\quad i=1,\ldots,n.
 \label{eq:model}
\end{equation}
Here $A_i\in\R^{d_1\times d_2}$ is the $i$-th sensing matrix and
$\epsilon_i$ is additive observation noise. The pairs $(A_i,\epsilon_i)$ are independent across $i$, and $A_i$ is independent
of $\epsilon_i$ within each pair.  We assume that $A_i$ is centered, isotropic in
Frobenius geometry, and uniformly sub-Gaussian: for every deterministic
$M\in\R^{d_1\times d_2}$,
\begin{equation}
 \E A_i=0,\quad
 \E[\ip{A_i}{M}A_i]=M,\quad
 \|\ip{A_i}{M}\|_{\psi_2}\le K_A\|M\|_F.
 \label{eq:sensing-assumption}
\end{equation}
The noise is centered and sub-Gaussian with scale proxy $\sigma$:
\begin{equation}
 \E\epsilon_i=0,\quad
 \|\epsilon_i\|_{\psi_2}\le K_\epsilon\sigma.
 \label{eq:noise-assumption}
\end{equation}
The constants $K_A,K_\epsilon\ge1$ are fixed throughout.

We represent the unknown matrix as $X=LR^\top$ using factor
width $k\ge r$, with $L\in\R^{d_1\times k}$ and
$R\in\R^{d_2\times k}$, and minimize
\begin{equation}
 \Phi_n(L,R)=\frac1{2n}\sum_{i=1}^n
 \bigl(\ip{A_i}{LR^\top}-y_i\bigr)^2
 +\frac18\|L^\top L-R^\top R\|_F^2.
 \label{eq:objective}
\end{equation}
The first term is empirical least squares.  The second term penalizes the noncompact rescaling symmetry $LR^\top=(LH)(RH^{-\top})^\top$ and vanishes at the balanced target factors $L^\star=[P\Sigma^{1/2},0]$ and $R^\star=[Q\Sigma^{1/2},0]$, padded to width $k$. The coefficient is fixed at $1/8$ since it will produce the exact symmetric lifting in \cref{lem:lifting}. Moreover, the distributional assumptions show that the population counterpart of \eqref{eq:objective} is, up to a constant independent of $(L,R)$,
\begin{equation}
 \Phi(L,R):=\E\Phi_n(L,R)-c_\epsilon
 =\frac12\|LR^\top-X^\star\|_F^2
  +\frac18\|L^\top L-R^\top R\|_F^2,
 \quad
 c_\epsilon:=\frac1{2n}\sum_{i=1}^n\E\epsilon_i^2.
 \label{eq:population-objective}
\end{equation}
Indeed, the cross term vanishes by centering and independence, and
\eqref{eq:sensing-assumption} gives
$\E\langle A_i,LR^\top-X^\star\rangle^2
=\|LR^\top-X^\star\|_F^2$.

Starting from $(L_0,R_0)$, gradient descent with a
constant step size $\eta$ performs
\begin{equation}
 L_{t+1}=L_t-\eta\nabla_L\Phi_n(L_t,R_t),\quad
 R_{t+1}=R_t-\eta\nabla_R\Phi_n(L_t,R_t).
 \label{eq:gd}
\end{equation}

Define the stacked factor, a balanced target factor, and its Gram matrix by
\begin{equation}
 Z_t=\begin{bmatrix}L_t\\R_t\end{bmatrix},\quad Z^\star=\begin{bmatrix}P\Sigma^{1/2}\\Q\Sigma^{1/2}\end{bmatrix},\quad
 M_+=Z^\star Z^{\star\top}=\begin{bmatrix}
 P\Sigma P^\top&X^\star\\X^{\star\top}&Q\Sigma Q^\top
 \end{bmatrix}.
 \label{eq:target}
\end{equation}

\subsection{Main results}
We now state the local convergence guarantees of factorized gradient descent for overparameterized asymmetric matrix sensing with a balancing regularization.

\begin{theorem}[Noisy asymmetric sensing with $k\ge r$]
\label{thm:main}
There exist positive constants
\[
 c_{\rm thm},\ c_{\rm prob},\ C_{\rm thm}>0,
\]
depending only on $K_A$ and $K_\epsilon$.  Suppose
\begin{align}
 \|Z_0Z_0^\top-M_+\|_2&\le c_{\rm thm}\sigma_r,\label{eq:init}\\
 0<\eta&\le c_{\rm thm}/\sigma_1,\label{eq:stepsize}\\
 n&\ge C_{\rm thm}\kappa^2D\max\left\{k^2,\frac{\sigma^2}{\sigma_r^2}\right\}.
 \label{eq:samples}
\end{align}
Then, with probability at least $1-D^{-c_{\rm prob}}$, simultaneously for every $t\ge0$,
\begin{equation}
 \|L_tR_t^\top-X^\star\|_2
 \le C_{\rm thm}\max\left\{\frac1{\eta t+1/\sigma_r},
 \eps_{\rm stat}\right\},
 \quad
 \eps_{\rm stat}:=\kappa\sigma\sqrt{\frac Dn}.
 \label{eq:mainrate}
\end{equation}
If $\eps_{\rm stat}>0$ and $t\ge 1/(\eta\eps_{\rm stat})$, then
\begin{equation}
 \|L_tR_t^\top-X^\star\|_2\le C_{\rm thm}\eps_{\rm stat},\quad
 \|L_tR_t^\top-X^\star\|_F^2
 \le 2C_{\rm thm}k\kappa^2\frac Dn\sigma^2.
 \label{eq:statistical}
\end{equation}
When $\sigma=0$, $\eps_{\rm stat}$ reduces to 0, and then the iterates converge to exact recovery.
\end{theorem}

In the exact rank case, we can further obtain a linear convergence result.
\begin{corollary}
\label{cor:matched}
Under the conditions of \cref{thm:main}, suppose in addition that $k=r$ or $D=2k$. Then there exist constants $c_{\rm lin},C_{\rm lin}>0$, depending only on $K_A$ and $K_\epsilon$, such that, for every $t\ge0$,
\begin{equation}
 \|L_tR_t^\top-X^\star\|_2
 \le C_{\rm lin}\max\left\{
 (1-c_{\rm lin}\eta\sigma_r)^t
 \|Z_0Z_0^\top-M_+\|_2,\ \eps_{\rm stat}\right\}.
 \label{eq:matchedrate}
\end{equation}
In particular, the noiseless iterates converge geometrically to exact recovery.
\end{corollary}

\begin{remark}
The two entries in the maximum in \eqref{eq:mainrate} have different origins. In the genuinely overspecified case, the first is the optimization error controlled by the quadratically contracting null block; the second is the statistical radius created by the observation noise. The sample condition is equivalent, up to numerical constants, to
$a_n:=k\sqrt{D/n}\lesssim1/\kappa$ and $\kappa b_n\lesssim\sigma_r$, where $b_n:=\sigma\sqrt{D/n}$.
\end{remark}

\section{Lifted geometry and population dynamics}
\label{sec:lifting}

We first remove the sampling fluctuation and analyze the population objective, for which the dynamics can be controlled block by block. The empirical objective will be analyzed as a perturbation of this population function.

\subsection{Exact symmetric lifting}

Define
\begin{equation}
 C^\star=\begin{bmatrix}0&X^\star\\X^{\star\top}&0\end{bmatrix},
 \quad
 J=\begin{bmatrix}I_{d_1}&0\\0&-I_{d_2}\end{bmatrix}.
 \label{eq:CJ}
\end{equation}

\begin{lemma}
\label{lem:lifting}
The population objective
\[
 \Phi(L,R)=\frac12\|LR^\top-X^\star\|_F^2
 +\frac18\|L^\top L-R^\top R\|_F^2
\]
satisfies
\begin{equation}
 \Phi(L,R)=\frac18\|ZZ^\top-2C^\star\|_F^2+\mathrm{constant},
 \quad
 \nabla_Z\Phi(Z)=\frac12(ZZ^\top-2C^\star)Z.
 \label{eq:lift}
\end{equation}
\end{lemma}

\begin{proof}
The cyclic trace identity gives
\begin{align*}
 \|L^\top L-R^\top R\|_F^2
 &=\|LL^\top\|_F^2+\|RR^\top\|_F^2-2\|LR^\top\|_F^2,\\
 \|ZZ^\top\|_F^2
 &=\|LL^\top\|_F^2+\|RR^\top\|_F^2+2\|LR^\top\|_F^2,\\
 \ip{ZZ^\top}{C^\star}&=2\ip{LR^\top}{X^\star}.
\end{align*}
By substitution, we have
\[
 \Phi(L,R)=\frac18\|ZZ^\top\|_F^2-\ip{LR^\top}{X^\star}
 +\frac12\|X^\star\|_F^2=\frac18\|ZZ^\top-2C^\star\|_F^2-\frac12\|X^\star\|_F^2.
\]
By direct computation, we obtain the gradient formula.
\end{proof}

\subsection{Spectral coordinates and Gram blocks}
Set
\begin{equation}
 U_+=\frac1{\sqrt2}\begin{bmatrix}P\\Q\end{bmatrix},\quad
 U_-=\frac1{\sqrt2}\begin{bmatrix}P\\-Q\end{bmatrix},\quad
 \Lambda=2\Sigma.
 \label{eq:eigenspaces}
\end{equation}
Then $U_+^\top U_-=0$ and
\begin{equation}
 2C^\star=U_+\Lambda U_+^\top-U_-\Lambda U_-^\top=:M_+-M_-.
 \label{eq:spectral}
\end{equation}
Indeed, $C^\star U_+=U_+\Sigma$ and $C^\star U_-=-U_-\Sigma$.  Since $\rank(C^\star)=2r$, the orthogonal complement of $\operatorname{span}(U_+,U_-)$ is $\ker(C^\star)$.  Choose $W$ so that $[U_+,U_-,W]$ is orthogonal, and decompose uniquely
\begin{equation}
 Z_t=U_+S_t+U_-N_t+WT_t,
 \quad
 H_t=Z_t^\top Z_t=S_t^\top S_t+N_t^\top N_t+T_t^\top T_t.
 \label{eq:decomposition}
\end{equation}
Here $S_t=U^\top_+Z_t,N_t=U^\top_-Z_t\in\R^{r\times k}$ and
$T_t=W^\top Z_t\in\R^{(D-2r)\times k}$. When $D=2r$, the $T_t$ row is vacuous and all blocks involving $T_t$ are zero. This then convention covers the case in which $\ker(C^\star)$ is trivial. Let $Z_{t}^+=Z_t-\eta\nabla_Z\Phi(Z_t)$ denote the gradient descent update for the population objective. Projecting the population update from \cref{lem:lifting} gives the exact coefficient recursions for gradient descent:
\begin{align}
 S_t^+&=S_t-\frac\eta2(S_tH_t-\Lambda S_t),\label{eq:Spop}\\
 N_t^+&=N_t-\frac\eta2(N_tH_t+\Lambda N_t),\label{eq:Npop}\\
 T_t^+&=T_t-\frac\eta2T_tH_t.\label{eq:Tpop}
\end{align}
The three blocks $\{U_+,U_-,W\}$ describe the positive, negative, and null spectral coordinates. The analysis below shows that $S_tS_t^\top\to \Lambda$, $N_tN_t^\top\to 0$, and $T_tT_t^\top\to 0$.
To verify the projection, let
$\mathcal U=[U_+,U_-,W]$ and $Y_t=\mathcal U^\top Z_t=[S_t^\top,N_t^\top,T_t^\top]^\top$.
Orthogonality gives
\[
 \mathcal U^\top Z_tZ_t^\top Z_t=Y_t(Z_t^\top Z_t)=Y_tH_t,
 \quad
 \mathcal U^\top(M_+-M_-)Z_t=
 \begin{bmatrix}\Lambda S_t\\-\Lambda N_t\\0\end{bmatrix}.
\]
Applying $\mathcal U^\top$ to
$Z_t^+=Z_t-\frac\eta2(Z_tZ_t^\top-M_++M_-)Z_t$ and reading off the three block rows yields \eqref{eq:Spop}--\eqref{eq:Tpop}.

Define six Gram blocks
\begin{align}
 A_t&=S_tS_t^\top-\Lambda,&B_t&=S_tN_t^\top,&C_t&=S_tT_t^\top,\nonumber\\
 E_t&=N_tN_t^\top,&F_t&=N_tT_t^\top,&G_t&=T_tT_t^\top,
 \label{eq:blocks}
\end{align}
Thus $A_t,B_t,E_t\in\R^{r\times r}$,
$C_t,F_t\in\R^{r\times(D-2r)}$, and
$G_t\in\R^{(D-2r)\times(D-2r)}$. Replacing $S_t,N_t,T_t$ with $S^+_t,N^+_t,T^+_t$, we call the result $H_t^+$, $H\in\{A,B,C,E,F,G\}$. Define also
\begin{align}
 \dist_t&=\max\{\|A_t\|_2,\|B_t\|_2,\|C_t\|_2,
 \|E_t\|_2,\|F_t\|_2,\|G_t\|_2\},\label{eq:D}\\
 \dist_t^{\rm nz}&=\max\{\|A_t\|_2,\|B_t\|_2,\|C_t\|_2,
 \|E_t\|_2,\|F_t\|_2\}.
 \label{eq:Dnz}
\end{align}
In the orthogonal basis $\mathcal U=[U_+,U_-,W]$,
\begin{equation}
 \mathcal U^\top(Z_tZ_t^\top-M_+)\mathcal U
 =\begin{bmatrix}A_t&B_t&C_t\\B_t^\top&E_t&F_t\\C_t^\top&F_t^\top&G_t\end{bmatrix}.
 \label{eq:blockmatrix}
\end{equation}
The block-row norm bound and compression to the upper-right asymmetric block imply
\begin{equation}
 \|Z_tZ_t^\top-M_+\|_2\le3\dist_t,
 \quad
 \|L_tR_t^\top-X^\star\|_2\le3\dist_t.
 \label{eq:productbound}
\end{equation}
Indeed, write the block matrix in \eqref{eq:blockmatrix} as $\mathcal E_t$ and
let $x=(x_1,x_2,x_3)$ be conformably partitioned.  Since every block of
$\mathcal E_t$ has norm at most $\dist_t$,
\[
 \|[\mathcal E_tx]_i\|_2\le \dist_t\sum_{j=1}^3\|x_j\|_2,
 \quad
 \|\mathcal E_tx\|_2^2
 \le3\dist_t^2\big(\sum_{j=1}^3\|x_j\|_2\big)^2
 \le9\dist_t^2\|x\|_2^2.
\]
Thus $\|\mathcal E_t\|_2\le3\dist_t$, which which proves the first inequality in \eqref{eq:productbound}. The matrix $L_tR_t^\top-X^\star$ is the upper-right block of $Z_tZ_t^\top-M_+$, so compression proves the second inequality.

\begin{remark}
Note that $JU_+=U_-$, $JU_-=U_+$ and $JW=WJ_0$, where $J_0=W^\top JW$. Substituting \eqref{eq:decomposition} into $Z_t^\top JZ_t$ gives
\begin{equation}
L_t^\top L_t-R_t^\top R_t
=Z_t^\top JZ_t
=S_t^\top N_t+N_t^\top S_t+T_t^\top J_0T_t.
\label{eq:balance-coordinates}
\end{equation}
Thus both $N_t$ and  $T_t$ may contribute to the imbalance.
\end{remark}

\subsection{Blockwise population contraction}
\label{sec:population}

The spectral coordinates reduce the population update to matrix
recursion \eqref{eq:Spop}-\eqref{eq:Tpop}.  The next lemma presents the local estimate: the five blocks $\{A_t,B_t,C_t,E_t,F_t\}$ contract linearly up to quadratic forcing, while the null Gram block $G_t$ contracts only through a quadratic term.

\begin{lemma}
\label{lem:population}
There are positive absolute constants $c_{\rm loc},\ c_{\rm step},\ c_{\rm pop},\ C_{\rm pop}>0$. Here $c_{\rm loc},\ c_{\rm step}\le1$. If $\dist_t\le c_{\rm loc}\sigma_r$  and $ 0<\eta\le c_{\rm step}/\sigma_1$
then
\begin{align}
 \max_{H=A,B,C,E,F}\|H_t^+\|_2
 &\le(1-c_{\rm pop}\eta\sigma_r)\dist_t^{\rm nz}
 +C_{\rm pop}\eta\dist_t^2,\label{eq:popnz}\\
 \|G_t^+\|_2&\le\|G_t\|_2-c_{\rm pop}\eta\|G_t\|_2^2.
 \label{eq:popg}
\end{align}
\end{lemma}

\begin{proof}
Since $S_tS_t^\top=\Lambda+A_t$, we have
\begin{equation}
 (2-c_{\rm loc})\sigma_rI_r\preceq S_tS_t^\top
 \preceq(2\sigma_1+c_{\rm loc}\sigma_r)I_r.
 \label{eq:Sbounds}
\end{equation}
The two inequalities follow from
$-\|A_t\|_2I_r\preceq A_t\preceq\|A_t\|_2I_r$ and the bounds $2\sigma_r I_r\preceq\Lambda\preceq2\sigma_1 I_r$. Moreover,
$\|N_tN_t^\top\|_2=\|E_t\|_2$ and
$\|T_tT_t^\top\|_2=\|G_t\|_2$, so both are at most
$c_{\rm loc}\sigma_r$.  Since $H_t=S_t^\top S_t+N_t^\top N_t+T_t^\top T_t$
is positive semidefinite,
\[
 \|H_t\|_2\le\|S_tS_t^\top\|_2+
 \|N_tN_t^\top\|_2+\|T_tT_t^\top\|_2
 \le2\sigma_1+3c_{\rm loc}\sigma_r\le C_H\sigma_1
\]
for an absolute $C_H>2$.  Choose
\begin{equation}
 c_{\rm step}\le\min\{C_H^{-1},1/2\}.
 \label{eq:stepchoice}
\end{equation}
Then
\begin{equation}
 \eta\|H_t\|_2\le1,\quad I-\eta H_t/2\succeq0,\quad
 0\preceq I-\eta\Lambda\preceq I.
 \label{eq:psdsteps}
\end{equation}

%We now derive every leading term from a single Gram identity. 
Define
\begin{equation}
Y_t=\begin{bmatrix}S_t\\N_t\\T_t\end{bmatrix},\quad
 \mathcal G_t=Y_tY_t^\top=\begin{bmatrix}K_{S,t}&B_t&C_t\\B_t^\top&E_t&F_t\\C_t^\top&F_t^\top&G_t\end{bmatrix},
 \quad K_{S,t}=\Lambda+A_t,
 \label{eq:biggram}
\end{equation}
and
\begin{equation}
 \mathcal L=\diag(\Lambda,-\Lambda,0),\quad
 \mathcal R_t=\mathcal G_t-\mathcal L =\begin{bmatrix}A_t&B_t&C_t\\B_t^\top&E_t+\Lambda&F_t\\C_t^\top&F_t^\top&G_t\end{bmatrix}.
 \label{eq:residual}
\end{equation}
Equations \eqref{eq:Spop}--\eqref{eq:Tpop} are equivalent to
$Y_t^+=(I-\eta\mathcal R_t/2)Y_t$, and therefore
\begin{equation}
 \mathcal G_t^+=\mathcal G_t
 -\frac{\eta}{2}
 \bigl(\mathcal R_t\mathcal G_t+\mathcal G_t\mathcal R_t\bigr)
 +\frac{\eta^2}{4}\mathcal R_t\mathcal G_t\mathcal R_t.
 \label{eq:gramidentity}
\end{equation}
For transparency, direct block multiplication of the first-order term gives
\begin{align*}
 [\mathcal R_t\mathcal G_t+\mathcal G_t\mathcal R_t]_{11}
 &=A_tK_{S,t}+K_{S,t}A_t
 +2B_tB_t^\top+2C_tC_t^\top,\\
 [\mathcal R_t\mathcal G_t+\mathcal G_t\mathcal R_t]_{12}
 &=\Lambda B_t+B_t\Lambda
 +2(A_tB_t+B_tE_t+C_tF_t^\top),\\
 [\mathcal R_t\mathcal G_t+\mathcal G_t\mathcal R_t]_{13}
 &=\Lambda C_t+2(A_tC_t+B_tF_t+C_tG_t),\\
 [\mathcal R_t\mathcal G_t+\mathcal G_t\mathcal R_t]_{22}
 &=\Lambda E_t+E_t\Lambda
 +2(B_t^\top B_t+E_t^2+F_tF_t^\top),\\
 [\mathcal R_t\mathcal G_t+\mathcal G_t\mathcal R_t]_{23}
 &=\Lambda F_t
 +2(B_t^\top C_t+E_tF_t+F_tG_t).
\end{align*}
Substituting these identities into \eqref{eq:gramidentity} gives
\begin{align}
 A^+_{t}&=A_t-\tfrac{\eta}{2}
 \bigl(A_tK_{S,t}+K_{S,t}A_t\bigr)
 -\eta(B_tB_t^\top+C_tC_t^\top)
 +R_{A,t},
 \label{eq:Arec}\\
 B^+_{t}&=B_t-\tfrac{\eta}{2}(\Lambda B_t+B_t\Lambda)
 -\eta(A_tB_t+B_tE_t+C_tF_t^\top)
 +R_{B,t},
 \label{eq:Brec}\\
 C^+_{t}&= C_t-\tfrac{\eta}{2}\Lambda C_t
 -\eta(A_tC_t+B_tF_t+C_tG_t)
 +R_{C,t},
 \label{eq:Crec}\\
 E^+_{t}&=P_\Lambda E_tP_\Lambda
 -\eta(B_t^\top B_t+E_t^2+F_tF_t^\top)
 +R_{E,t},
 \label{eq:Erec}\\
 F^+_{t}&=F_t-\tfrac{\eta}{2}\Lambda F_t-\eta(B_t^\top C_t+E_tF_t+F_tG_t)
 +R_{F,t},
 \label{eq:Frec}
\end{align}
where $P_\Lambda=I_r-\eta\Lambda/2$ and
\begin{align}
R_{A,t}&=\tfrac{\eta^2}{4}
 [\mathcal R_t\mathcal G_t\mathcal R_t]_{11},
 &
 R_{B,t}&=
 \tfrac{\eta^2}{4}
 [\mathcal R_t\mathcal G_t\mathcal R_t]_{12},
 \nonumber\\
 R_{C,t}&=
 \tfrac{\eta^2}{4}
 [\mathcal R_t\mathcal G_t\mathcal R_t]_{13},
 &
 R_{E,t}&=
 \tfrac{\eta^2}{4}
 \left(
 [\mathcal R_t\mathcal G_t\mathcal R_t]_{22}
 -\Lambda E_t\Lambda
 \right),
 \nonumber\\
 R_{F,t}&=\tfrac{\eta^2}{4}
 [\mathcal R_t\mathcal G_t\mathcal R_t]_{23}.
 \label{eq:remainders}
\end{align}
The subtraction in $R_{E,t}$ removes the $\Lambda E_t\Lambda$ term already included in $P_\Lambda E_tP_\Lambda$.  The complete term-by-term expansions are recorded in \cref{app:remainders}.

Put $\rho=\dist_t$. % In $\mathcal G$, only $K_S$ can have norm of order $\sigma_1$; in $\mathcal R$, only $E+\Lambda$ can.  Apart from the extracted $\Lambda E\Lambda$ term, every product in \eqref{eq:remainders} contains at least two blocks from $\{A_t,B_t,C_t,E_t,F_t,G_t\}$. and at most one factor of order $\sigma_1$. Since $\rho\le c_{\rm loc}\sigma_r\le c_{\rm loc}\sigma_1$, 
From \cref{app:remainders}, there exists a constant $C$ such that
\begin{equation}
 \max_{H\in\{A,B,C,E,F\}}\|R_{H,t}\|_2
 \le C\eta^2\sigma_1\rho^2\le C\eta\rho^2.
 \label{eq:remainderbound}
\end{equation}

For every $X$ from $\{A_t,B_t,C_t,E_t,F_t,G_t\}$, we have
\begin{align}
 \left\|X-\tfrac\eta2(\Lambda X+X\Lambda)\right\|_2
 &\le(1-2\eta\sigma_r)\|X\|_2,\label{eq:twosided}\\
 \left\|X-\tfrac\eta2\Lambda X\right\|_2
 &\le(1-\eta\sigma_r)\|X\|_2,\label{eq:onesided}\\
 \|P_\Lambda XP_\Lambda\|_2&\le(1-\eta\sigma_r)^2\|X\|_2.
 \label{eq:sandwich}
\end{align}
The left side of \eqref{eq:twosided} equals one half of
$(I-\eta\Lambda)X+X(I-\eta\Lambda)$, and
$\|I-\eta\Lambda\|_2=1-2\eta\sigma_r$.  Likewise,
$X-(\eta/2)\Lambda X=P_\Lambda X$ and
$\|P_\Lambda\|_2=1-\eta\sigma_r$; applying this bound on both sides of
$P_\Lambda XP_\Lambda$ proves \eqref{eq:onesided}--\eqref{eq:sandwich}.
Substituting \(K_{S,t}=\Lambda+A_t\) into \eqref{eq:Arec} yields
\[
 A_{t}^+
 =
 A_t-\tfrac{\eta}{2}(A_t\Lambda+\Lambda A_t)
 -\eta(A_t^2+B_tB_t^\top+C_tC_t^\top)
 +R_{A,t}.
\]
We now bound the five recursions separately.  Write
$x=\dist_t^{\rm nz}$ and $\rho=\dist_t$.  Submultiplicativity of the operator
norm, \eqref{eq:twosided}--\eqref{eq:sandwich}, and
\eqref{eq:remainderbound} give
\begin{align}
\|A_{t}^+\|_2
 &\le
 (1-2\eta\sigma_r)\|A_t\|_2
 +\eta\bigl(
 \|A_t\|_2^2+\|B_t\|_2^2+\|C_t\|_2^2
 \bigr)
 \notag\\
 &\quad+C\eta\rho^2
 \le
 (1-2\eta\sigma_r)x+C\eta\rho^2,
 \label{eq:Abound}\\
 \|B_{t}^+\|_2
 &\le
 (1-2\eta\sigma_r)\|B_t\|_2
 +\eta\bigl(
 \|A_t\|_2\|B_t\|_2
 +\|B_t\|_2\|E_t\|_2
 +\|C_t\|_2\|F_t\|_2
 \bigr)
 \notag\\
 &\quad+C\eta\rho^2
 \le
 (1-2\eta\sigma_r)x+C\eta\rho^2,
 \label{eq:Bbound}\\
 \|C_{t}^+\|_2
 &\le
 (1-\eta\sigma_r)\|C_t\|_2
 +\eta\bigl(
 \|A_t\|_2\|C_t\|_2
 +\|B_t\|_2\|F_t\|_2
 +\|C_t\|_2\|G_t\|_2
 \bigr)
 \notag\\
 &\quad+C\eta\rho^2
 \le
 (1-\eta\sigma_r)x+C\eta\rho^2,
 \label{eq:Cbound}\\
 \|E_{t}^+\|_2
 &\le
 (1-\eta\sigma_r)^2\|E_t\|_2
 +\eta\bigl(
 \|B_t\|_2^2+\|E_t\|_2^2+\|F_t\|_2^2
 \bigr)
 \notag\\
 &\quad+C\eta\rho^2
 \le
 (1-\eta\sigma_r)x+C\eta\rho^2,
 \label{eq:Ebound}\\
 \|F_{t}^+\|_2
 &\le
 (1-\eta\sigma_r)\|F_t\|_2
 +\eta\bigl(
 \|B_t\|_2\|C_t\|_2
 +\|E_t\|_2\|F_t\|_2
 +\|F_t\|_2\|G_t\|_2
 \bigr)
 \notag\\
 &\quad+C\eta\rho^2
 \le
 (1-\eta\sigma_r)x+C\eta\rho^2.\label{eq:Fbound}
\end{align}
In \eqref{eq:Ebound} we used $(1-u)^2\le1-u$ for
$u=\eta\sigma_r\in[0,1]$.  Taking the maximum of
\eqref{eq:Abound}--\eqref{eq:Fbound} proves \eqref{eq:popnz}; any fixed $c_{\rm pop}\le 1$ and $C_{\rm pop}\ge \frac{23}{4}$ satisfy this bound.

Finally, \eqref{eq:Tpop} gives
\begin{align*}
 G_{t}^+=T_t(I-\eta H_t/2)^2T_t^\top=
 G_t-\eta T_tH_tT_t^\top
 +\frac{\eta^2}{4}T_tH_t^2T_t^\top.
\end{align*}
Because $0\preceq\eta H_t \preceq I$, spectral calculus gives
$H_t^2\preceq\|H_t\|_2H_t\preceq\eta^{-1}H_t$.  Therefore the last term is at most
$(\eta/4)T_tH_tT_t^\top$ in Loewner order.  Moreover,
\[
 T_tH_tT_t^\top=T_t\bigl(S_t^\top S_t+N_t^\top N_t+T_t^\top T_t\bigr)T_t^\top\succeq T_t(T_t^\top T_t)T_t^\top=G_t^2.
\]
It follows that
\begin{equation}
 0\preceq G_t^+\preceq G_t-\frac{3\eta}{4}G_t^2.
 \label{eq:Gloewner}
\end{equation}
Choose $c_{\rm pop}\le3/4$ so that $2c_{\rm pop}\eta\|G\|_2\le2c_{\rm pop}c_{\rm step}c_{\rm loc}\le1$. The scalar map
$q\mapsto q-c_{\rm pop}\eta q^2$ is increasing on
$[0,\|G_t\|_2]$.  Weyl monotonicity and functional calculus applied to
\eqref{eq:Gloewner} now yield
\[
 \|G_t^+\|_2\le\lambda_{\max}(G_t-c_{\rm pop}\eta G_t^2)
 =\|G_t\|_2-c_{\rm pop}\eta\|G_t\|_2^2,
\]
which is \eqref{eq:popg}.

\end{proof}

\subsection{A tight sublinear example}
\label{sec:sharp-example}

The following noiseless population example shows that for any local neighborhood of the true matrix, there exists a point in this neighborhood, such that the iterations started from this point exhibit a tight sublinear convergence rate.

\begin{proposition}
\label{prop:sharp-sublinear}
Let $d_1=d_2=k=2$, $r=1$, and
$X^\star=\operatorname{diag}(\sigma_\star,0)$ with $\sigma_\star>0$.  Apply gradient descent with step size $\eta>0$ to the population objective in \eqref{eq:lift}, initialized at
\begin{equation}
 L_0=R_0=\operatorname{diag}(\sqrt{\sigma_\star},\frac{\alpha_0}{\sqrt{2}}), \quad \alpha_0>0.
 \label{eq:sharp-init}
\end{equation}
Define $e_t=\|L_tR_t^\top-X^\star\|_2$.  If $0<\eta e_0\le1/2$, then, for every $t\ge0$,
\begin{equation}
 \frac{1}{e_0^{-1}+6\eta t}
 \le e_t\le
 \frac{1}{e_0^{-1}+2\eta t}.
 \label{eq:sharp-rate}
\end{equation}
Consequently, $e_t=\Theta(1/(\eta t))$.  %In particular, there are no constants $C<\infty$ and $\rho\in(0,1)$ for which $e_t\le C\rho^t$ for every $t$. The initialization can be chosen arbitrarily close to the lifted target, so this failure of a geometric rate occurs inside every local neighborhood of $M_+$.
\end{proposition}

\begin{proof}
We prove by induction that each iteration has the following form:
\begin{equation}
 L_j=R_j=\operatorname{diag}(\sqrt{\sigma_\star},\frac{\alpha_j}{\sqrt{2}}).
 \label{eq:sharp-form}
\end{equation}
The claim holds at $j=0$. Suppose that it holds at $j=t$. Then $L_t^\top L_t=R_t^\top R_t$, so the balancing penalty and both of its
gradients vanish.  Moreover,
\begin{equation}
 L_tR_t^\top-X^\star
 =\frac{\alpha_t^2}{2}\operatorname{diag}(0,1).
 \label{eq:sharp-error}
\end{equation}
%For the population least-squares term, the two factor gradients are $(L_tR_t^\top-X^\star)R_t$ and $(L_tR_t^\top-X^\star)^\top L_t$, respectively.  Their first columns are zero, whereas their second columns both equal $[0;\frac{\alpha_t^3}{2\sqrt2}]$. 
By direct computation, one gradient step preserves \eqref{eq:sharp-form} and gives the exact scalar recursion
\begin{equation}
 \alpha_{t+1}=\alpha_t\big(1-\frac{\eta\alpha_t^2}{2}\big).
 \label{eq:sharp-alpha}
\end{equation}
By \eqref{eq:sharp-error}, $e_t=\alpha_t^2/2$, and hence
\begin{equation}
 e_{t+1}=e_t(1-\eta e_t)^2.
 \label{eq:sharp-scalar}
\end{equation}
Because $0<\eta e_0\le1/2$, this recursion shows inductively that
$0<e_{t+1}<e_t$ and $0<\eta e_t\le1/2$ for all $t$.  Taking reciprocals in
\eqref{eq:sharp-scalar} yields
\begin{equation}
 \frac1{e_{t+1}}-\frac1{e_t}
 =\eta\frac{2-\eta e_t}{(1-\eta e_t)^2}.
 \label{eq:sharp-reciprocal}
\end{equation}
For $x\in[0,1/2]$, we have $2\le\frac{2-x}{(1-x)^2}\le6$.
Applying with $x=\eta e_t$ and summing
\eqref{eq:sharp-reciprocal} from $0$ to $t-1$ gives
\[
 e_0^{-1}+2\eta t\le e_t^{-1}
 \le e_0^{-1}+6\eta t.
\]
Inverting the positive quantities proves \eqref{eq:sharp-rate}.

%The lower bound in \eqref{eq:sharp-rate} rules out $e_t\le C\rho^t$. Note that $e_{t+1}/e_t=(1-\eta e_t)^2\to1$, which is incompatible with an asymptotic $Q$-linear contraction factor strictly below one.

%For completeness, stack the factors and put $w=2^{-1/2}[0,1,0,1]^\top$.  The first column of $Z_0$ generates $M_+$ exactly, while the second is $\alpha_0w$.  Therefore \[Z_0Z_0^\top-M_+=\alpha_0^2ww^\top,\quad\|Z_0Z_0^\top-M_+\|_2=\alpha_0^2=2e_0. \] Taking $\alpha_0$ arbitrarily small places the trajectory inside any prescribed local basin. The redundant vector $w\in\ker(C^\star)$ and is orthogonal in factor coordinates to the signal column.%, so the construction is precisely an independent null-block trajectory.
\end{proof}

\section{Empirical perturbations and the master recursion}
\label{sec:concentration}

The proof now passes from deterministic geometry to finite-sample dynamics. We first isolate the empirical-gradient error and establish a uniform bound. We then project that perturbation into the same spectral coordinates and combine it with the population contraction.

\subsection{Uniform empirical-gradient control}

Define
\begin{equation}
\cA_n^*\cA_n(M)=\frac1n\sum_{i=1}^n\ip{A_i}{M}A_i,\quad W_n=\frac1n\sum_{i=1}^n\epsilon_iA_i.
\label{eq:operator}
\end{equation}
Frobenius isotropy in \eqref{eq:sensing-assumption} implies
$\E\cA_n^*\cA_n(M)=M$.  With
$\err_t=L_tR_t^\top-X^\star$, let
\begin{equation}
 \Delta_t=(\cA_n^*\cA_n-I)(\err_t)-W_n,\quad \widetilde\Delta_t=\begin{bmatrix}0&\Delta_t\\\Delta_t^\top&0\end{bmatrix}.
\label{eq:Delta}
\end{equation}
Then $\|\widetilde\Delta_t\|_2=\|\Delta_t\|_2$, and direct differentiation of \eqref{eq:objective} gives the exact empirical update
\begin{equation}
 Z_{t+1}=Z_t-\eta\left\{\frac12(Z_tZ_t^\top-M_++M_-)Z_t
 +\widetilde\Delta_tZ_t\right\}.
 \label{eq:empiricalupdate}
\end{equation}

We next prove the uniform event that controls $\Delta_t$ at every adaptive iterate.

The following proposition records the standard sub-Gaussian
rank-restricted isometry estimate in the normalization used throughout this paper. 
\begin{proposition}[Sub-Gaussian matrix RIP]
\label{prop:dirksen-rip}
There is an absolute constant $C_{\rm D}>0$ such that, for every integer
$1\le q\le\min\{d_1,d_2\}$ and every $0<\delta,\zeta<1$,
\begin{equation}
 n\ge C_{\rm D}K_A^4\delta^{-2}
 \max\{q(D+1),\log(\zeta^{-1})\}
 \label{eq:dirksen-samples}
\end{equation}
implies, with probability at least $1-\zeta$,
\begin{equation}
\delta_q^{\rm RIP}(\cA_n):=\sup_{\substack{\rank(X)\le q\\\|X\|_F=1}}\left|\|\cA_n(X)\|_2^2-1\right|\le\delta,\quad \cA_n(X):=\frac1{\sqrt n}(\ip{A_i}{X})_{i=1}^n.
\label{eq:dirksen-rip}
\end{equation}
\end{proposition}

\begin{proof}
%Vectorization identifies $\R^{d_1\times d_2}$, equipped with the Frobenius norm, with a Euclidean Hilbert space.  

By \eqref{eq:sensing-assumption}, each coordinate of $\cA_n$ satisfies
\[
 \|[\cA_n(X)-\cA_n(Y)]_i\|_{\psi_2}
 \le \frac{K_A}{\sqrt n}\|X-Y\|_F.
\]
Thus the parameter $\alpha$ in Dirksen's definition of a sub-Gaussian map is
$\alpha=K_A^2$.  Dirksen's rank-RIP result
\cite[Example~5.7]{dirksen2016}, obtained from its master Theorem~4.8 and the
covering estimate for the unit-Frobenius rank-$q$ class,
then gives \eqref{eq:dirksen-samples}--\eqref{eq:dirksen-rip}.  In particular,
the fourth power $K_A^4$ is the quantity $\alpha^2$ in that theorem.
\end{proof}

The next lemma shows the probabilistic and geometric tools used in the perturbation analysis.
\begin{lemma}
\label{lem:tools}
There is a fixed absolute constant $C_{\rm B}\ge1$ with the following
properties.
\begin{enumerate}[label=(\roman*)]
 \item The unit sphere in $\R^d$ has an $\eps$-net of cardinality at most
 $(1+2/\eps)^d$.
 \item If $\mathcal N_1,\mathcal N_2$ are $\eps$-nets of the two unit
 spheres and $\eps<1/2$, then
 \[
  \|B\|_2\le(1-2\eps)^{-1}
  \max_{u\in\mathcal N_1,v\in\mathcal N_2}|u^\top Bv|.
 \]
 \item If the independent centered variables $X_i$ obey
 $\|X_i\|_{\psi_1}\le K$, then
 \[
  \Prob\!\left(\left|\frac1n\sum_iX_i\right|>
  C_{\rm B}K\left(\sqrt{\frac xn}+\frac xn\right)\right)\le2e^{-x}.
 \]
 \item If $U,V$ are sub-Gaussian, then $UV$ is sub-exponential, and
 $$\|UV\|_{\psi_1}\le\|U\|_{\psi_2}\|V\|_{\psi_2}.$$
\end{enumerate}
\end{lemma}

\begin{proof}
For (i), take a maximal $\eps$-separated subset $\mathcal N$ of the unit
sphere.  Maximality makes $\mathcal N$ an $\eps$-net.  The Euclidean balls
$B(x,\eps/2)$, $x\in\mathcal N$, are disjoint and contained in
$B(0,1+\eps/2)$.  Comparing volumes gives
\[
 |\mathcal N|(\eps/2)^d\operatorname{vol}B(0,1)
 \le(1+\eps/2)^d\operatorname{vol}B(0,1),
\]
and hence $|\mathcal N|\le(1+2/\eps)^d$.

For (ii), choose $u_0,v_0$ in the respective nets with
$\|u-u_0\|_2,\|v-v_0\|_2\le\eps$.  Then
\begin{align*}
 |u^\top Bv|
 &\le |u_0^\top Bv_0|+|(u-u_0)^\top Bv|+|u_0^\top B(v-v_0)|\\
 &\le |u_0^\top Bv_0|+2\eps\|B\|_2.
\end{align*}
Taking the supremum over unit $u,v$ and moving the final term to the left
proves the claim.

Part (iii) follows from \cite[Theorem~2.9.1]{vershynin2018}. For (iv), see \cite[Lemma~2.8.6]{vershynin2018}.
\end{proof}

We now combine \cref{prop:dirksen-rip} with the tools above to construct the uniform concentration event used in the remainder of the proof in this section.  
\begin{lemma}[Uniform low-rank concentration]
\label{lem:concentration}
There are fixed constants $C_{\rm samp}\ge1$, $C_{\rm op}>0$, and
$C_{\rm noise}>0$ such that, if $n\ge C_{\rm samp}kD$, then with
probability at least $1-D^{-1}$, simultaneously for every radius $\rho>0$,
\begin{align}
 \sup_{\substack{\rank(M)\le2k\\\|M\|_2\le \rho}}
 \|(\cA_n^*\cA_n-I)(M)\|_2&\le C_{\rm op}a_n\rho,
 &a_n&=k\sqrt{\frac Dn},\label{eq:uniform}\\
 \|W_n\|_2&\le C_{\rm noise}b_n,&b_n&=\sigma\sqrt{\frac Dn}.
 \label{eq:noise}
\end{align}
Here $C_{\rm samp}$ and $C_{\rm op}$ depend only on $K_A$, and
$C_{\rm noise}$ depends only on $K_A$ and $K_\epsilon$.
\end{lemma}

\begin{proof}
Set $s=2k$ and $q=\min\{s+1,d_1,d_2\}$.  Let $C_{\rm D}$ be the absolute constant in
\cref{prop:dirksen-rip}, and fix once and for all
\begin{equation}
 C_{\rm rip}:=\max\{1,\sqrt{C_{\rm D}}\},\quad
 C_{\rm samp}:=\max\{1,72C_{\rm rip}^2K_A^4\},\quad
 C_{\rm op}:=3C_{\rm rip}K_A^2.
 \label{eq:ripconstants}
\end{equation}
Because $s\ge2$ and $D\ge2$, we have
$q\le3s/2$, $D+1\le3D/2$, and
$q(D+1)\ge\log(2D)$.  Under $n\ge C_{\rm samp}kD$, the deterministic
threshold
\begin{equation}
 \bar\delta_q:=C_{\rm rip}K_A^2
 \sqrt{\frac{q(D+1)}n}
 \label{eq:ripthreshold}
\end{equation}
satisfies $0<\bar\delta_q\le1/4$.  Applying
\cref{prop:dirksen-rip} with $\delta=\bar\delta_q$ and
$\zeta=(2D)^{-1}$ gives
\begin{equation}
 \delta_q^{\rm RIP}(\cA_n)\le\bar\delta_q
 \label{eq:ripevent}
\end{equation}
with failure probability at most $(2D)^{-1}$.

Write $\mathcal T=\cA_n^*\cA_n-I$ and $\mathcal B(M,N)=\ip{\mathcal T(M)}{N}$, then $\mathcal B$ is a symmetric bilinear form.  On \eqref{eq:ripevent},
\begin{equation}
 |\mathcal B(X,X)|\le\bar\delta_q\|X\|_F^2
 \quad\text{whenever }\rank(X)\le q.
 \label{eq:ripquadratic}
\end{equation}
Fix a nonzero matrix $M$ with $\rank(M)\le s$ and a rank-one matrix
$N=uv^\top$ with $\|u\|_2=\|v\|_2=1$.  Set
\[
 a=\sqrt{\frac{\|N\|_F}{\|M\|_F}},\quad
 \widehat M=aM,\quad \widehat N=a^{-1}N.
\]
Then $\|\widehat M\|_F=\|\widehat N\|_F
=\sqrt{\|M\|_F\|N\|_F}$ and
$\mathcal B(\widehat M,\widehat N)=\mathcal B(M,N)$.  Since
$\rank(\widehat M\pm\widehat N)\le
\min\{s+1,d_1,d_2\}=q$, polarization and
\eqref{eq:ripquadratic} give
\begin{align*}
 4|\mathcal B(M,N)|
 &\le |\mathcal B(\widehat M+\widehat N,\widehat M+\widehat N)|
+|\mathcal B(\widehat M-\widehat N, \widehat M-\widehat N)|\\ &\le\bar\delta_q\bigl(\|\widehat M+\widehat N\|_F^2 +\|\widehat M-\widehat N\|_F^2\bigr)=4\bar\delta_q\|M\|_F\|N\|_F.
\end{align*}
The same conclusion is trivial if $M=0$.  Since $\|uv^\top\|_F=1$, we
have proved
\begin{equation}
 |\ip{\mathcal T(M)}{uv^\top}|
 \le\bar\delta_q\|M\|_F
 \quad\text{for }\rank(M)\le s, ~\|u\|_2=\|v\|_2=1.
 \label{eq:restricted-orthogonality}
\end{equation}

Maximizing \eqref{eq:restricted-orthogonality} over unit vectors $u,v$
yields
\begin{align}
 \|\mathcal T(M)\|_2
 &=\sup_{\|u\|_2=\|v\|_2=1}
   |\ip{\mathcal T(M)}{uv^\top}|\notag\\
 &\le\bar\delta_q\|M\|_F
 \le\sqrt{s}\,\bar\delta_q\|M\|_2
 \le C_{\rm op}k\sqrt{\frac Dn}\|M\|_2.
 \label{eq:rip-to-operator}
\end{align}
The final inequality uses $q\le3s/2$, $D+1\le3D/2$, and $s=2k$.
Taking the supremum over $\|M\|_2\le \rho$ proves \eqref{eq:uniform}
simultaneously for every $\rho>0$ on the single RIP event.

For the noise term, fix the additional constants
\begin{equation}
 K_W:=K_AK_\epsilon,\quad
 H_W:=\log9,\quad \gamma_W:=H_W+2,
 \quad
 C_{\rm noise}:=2C_{\rm B}K_W(\sqrt{\gamma_W}+\gamma_W).
 \label{eq:noiseconstants}
\end{equation}
For fixed unit vectors $u,v$, independence of $\epsilon_i$ and $A_i$, together
with centering, gives
\[
 \E[\epsilon_i u^\top A_iv]
 =\E\epsilon_i\,\E[u^\top A_iv]=0.
\]
The Orlicz product bound gives
$\|\epsilon_i u^\top A_iv\|_{\psi_1}\le K_W\sigma$.  Take
$1/4$-nets of the two spheres.  Their joint log-cardinality is at most
$H_WD$.  Bernstein with $x=\gamma_WD$, followed by the union bound and
\cref{lem:tools} (iii), gives failure probability at most $2e^{-2D}$ and
\[
 \|W_n\|_2
 \le2C_{\rm B}K_W\sigma\left(
 \sqrt{\frac{\gamma_WD}{n}}+\frac{\gamma_WD}{n}\right)
 \le C_{\rm noise}\sigma\sqrt{\frac Dn},
\]
where the last step uses $n\ge C_{\rm samp}kD\ge D$.  Since $D\ge2$,
the intersection of the RIP and noise events has failure probability at most
\[
 (2D)^{-1}+2e^{-2D}\le D^{-1}.
\]
\end{proof}

Set $C_\Delta:=3C_{\rm op}+C_{\rm noise}$.  At every iterate,
$\rank(\err_t)\le k+r\le2k$ and \eqref{eq:productbound} holds.  Thus, on
the single event of \cref{lem:concentration},
\begin{equation}
 \delta_t:=\|\Delta_t\|_2\le C_\Delta(a_n\dist_t+b_n)
 \quad\text{for every }t\ge0.
 \label{eq:Deltabound}
\end{equation}
Indeed, linearity and the triangle inequality give
\begin{align*}
 \delta_t
 &\le\|(\cA_n^*\cA_n-I)(\err_t)\|_2+\|W_n\|_2\le C_{\rm op}a_n\|\err_t\|_2+C_{\rm noise}b_n\\
 &\le3C_{\rm op}a_n\dist_t+C_{\rm noise}b_n\le C_\Delta(a_n\dist_t+b_n).
\end{align*}

\subsection{Perturbed block dynamics and the master recursion}
\label{sec:master}

We next transfer the population block-wise contraction to the empirical gradient-descent dynamics. Let $U_0=W$ and define
\begin{equation}
 \Delta_{\alpha\beta,t}=U_\alpha^\top\widetilde\Delta_tU_\beta,
 \quad \alpha,\beta\in\{+,-,0\}.
 \label{eq:Deltablocks}
\end{equation}
Every such block has 2-norm at most $\delta_t$. Decompose the projected perturbations on three subspaces as $\operatorname{col}(U_+),\operatorname{col}(U_-),\operatorname{col}(W)$ as
\begin{align}
P_{S,t}&:=U_+^\top\widetilde\Delta_tZ_t=\Delta_{++,t}S_t+\Delta_{+-,t}N_t+\Delta_{+0,t}T_t,\nonumber\\ P_{N,t}&:=U_-^\top\widetilde\Delta_tZ_t=\Delta_{-+,t}S_t+\Delta_{--,t}N_t+\Delta_{-0,t}T_t,\nonumber\\ P_{T,t}&:=W^\top\widetilde\Delta_tZ_t=\Delta_{0+,t}S_t+\Delta_{0-,t}N_t+\Delta_{00,t}T_t.
 \label{eq:Pblocks}
\end{align}
Projection of \eqref{eq:empiricalupdate} yields
\begin{equation}
 S_{t+1}=S_t^+-\eta P_{S,t},\quad
 N_{t+1}=N_t^+-\eta P_{N,t},\quad
 T_{t+1}=T_t^+-\eta P_{T,t}.
 \label{eq:perturbedcoeff}
\end{equation}
For any two coefficient rows $X,Y$ with perturbations $P_X,P_Y$, the elementary
identity
\[
 (X^+-\eta P_X)(Y^+-\eta P_Y)^\top
 =X^+Y^{+\top}-\eta(X^+P_Y^\top+P_XY^{+\top})
 +\eta^2P_XP_Y^\top
\]
holds.  Apply it to $(X,Y)=(S_t,S_t),(S_t,N_t),(S_t,T_t),(N_t,N_t),(N_t,T_t),(T_t,T_t)$ and subtract
$\Lambda$ only in the $(S_t,S_t)$ block.  This gives the exact six-block identities
\begin{align}
 A_{t+1}&=A_t^+-\eta(S_t^+P_{S,t}^\top+P_{S,t}S_t^{+\top}) +\eta^2P_{S,t}P_{S,t}^\top,\nonumber\\
 B_{t+1}&=B_t^+-\eta(S_t^+P_{N,t}^\top+P_{S,t}N_t^{+\top}) +\eta^2P_{S,t}P_{N,t}^\top,\nonumber\\
 C_{t+1}&=C_t^+-\eta(S_t^+P_{T,t}^\top+P_{S,t}T_t^{+\top})
+\eta^2P_{S,t}P_{T,t}^\top,\nonumber\\
 E_{t+1}&=E_t^+-\eta(N_t^+P_{N,t}^\top+P_{N,t}N_t^{+\top}) +\eta^2P_{N,t}P_{N,t}^\top,\nonumber\\
 F_{t+1}&=F_t^+-\eta(N_t^+P_{T,t}^\top+P_{N,t}T_t^{+\top})
+\eta^2P_{N,t}P_{T,t}^\top,\nonumber\\
 G_{t+1}&=G_t^+-\eta(T_t^+P_{T,t}^\top+P_{T,t}T_t^{+\top})
 +\eta^2P_{T,t}P_{T,t}^\top.
 \label{eq:sixblocks}
\end{align}
Inside $\dist_t\le c_{\rm loc}\sigma_r$, \eqref{eq:productbound} and
$M_+=U_+\Lambda U_+^\top$ give
\[
 \|Z_t\|_2^2=\|Z_tZ_t^\top\|_2
 \le\|M_+\|_2+\|Z_tZ_t^\top-M_+\|_2
 \le2\sigma_1+3c_{\rm loc}\sigma_r\le C_H\sigma_1
\]
for an absolute constant $C_H>0$.
Thus $\|Z_t\|_2\le \sqrt{C_H}\sqrt{\sigma_1}$.  Moreover,
$Y_t^+=(I-\eta\mathcal R_t/2)Y_t$,
$\|\mathcal R_t\|_2\le 2\sigma_1+3c_{\rm loc}\sigma_r\le C_H\sigma_1$, and
$\eta\sigma_1\le c_{\rm step}$, so there exist constants $C_1,C_2>0$ such that
\begin{equation}
 \max\{\|S_t^+\|_2,\|N_t^+\|_2,\|T_t^+\|_2\}
 \le C_1\sqrt{\sigma_1},\quad
 \max\{\|P_{S,t}\|_2,\|P_{N,t}\|_2,\|P_{T,t}\|_2\}
 \le C_2\delta_t\sqrt{\sigma_1}.
 \label{eq:coefficientbounds}
\end{equation}
For the first estimate in \eqref{eq:coefficientbounds}, each coefficient block
is a row compression of $Y_t^+$ and hence has norm at most $\|Y_t^+\|_2$;
orthogonality gives $\|Y_t\|_2=\|Z_t\|_2$, so
\[
\|Y_t^+\|_2\le(1+\eta\|\mathcal R_t\|_2/2)\|Y_t\|_2
\le (1+\tfrac{\eta \sigma_1}{2}C_H)\sqrt{C_H}\sqrt{\sigma_1}
\le C_1\sqrt{\sigma_1},
\]
where $C_1=(1+\frac{c_{\rm step}}{2}C_H)\sqrt{C_H}$. For the second estimate, for example,
\[
 \|P_{S,t}\|_2
 \le\delta_t(\|S_t\|_2+\|N_t\|_2+\|T_t\|_2)
 \le3\delta_t\|Y_t\|_2
 \le C_2\delta_t\sqrt{\sigma_1},
\]
where $C_2=3\sqrt{C_H}$. The same argument applies to the $N$ and $T$ rows.
Thus, in each of the first five equations of \eqref{eq:sixblocks}, the first- and second-order empirical corrections are at most
\begin{equation}
 C\eta\sigma_1\delta_t
 \quad\text{and}\quad C\eta^2\sigma_1\delta_t^2,
 \label{eq:crudecorrections}
\end{equation}
where $C=\max\{2C_1C_2,C_2^2,6+3C_Hc_{\rm step}\}$.
For example, the two corrections in the $A$ equation satisfy
\begin{align*}
 \eta\|S_t^+P_{S,t}^\top+P_{S,t}S_t^{+\top}\|_2
 &\le2\eta\|S_t^+\|_2\|P_{S,t}\|_2
 \le C\eta\sigma_1\delta_t,\\
 \eta^2\|P_{S,t}P_{S,t}^\top\|_2
 &=\eta^2\|P_{S,t}\|_2^2
 \le C\eta^2\sigma_1\delta_t^2.
\end{align*}
For a cross block, such as $B$, the same calculation is applied separately to $S_t^+P_{N,t}^\top$ and $P_{S,t}N_t^{+\top}$; the other three nonzero blocks are identical after relabeling.  This proves \eqref{eq:crudecorrections} for each of the first five equations.
%, rather than only for their maximum.

The null block admits a sharper first-order bound.  Expanding \eqref{eq:Pblocks},
\[
 T_t^+P_{T,t}^\top=T_t^+S_t^\top\Delta_{0+,t}^\top
 +T_t^+N_t^\top\Delta_{0-,t}^\top
 +T_t^+T_t^\top\Delta_{00,t}^\top.
\]
Set $K_{S,t}=S_tS_t^\top=\Lambda+A_t$. Using
$T_t^+=T_t(I-\eta H_t/2)$ and \eqref{eq:biggram},
\begin{align*}
 T_t^+S_t^\top&=C_t^\top-\frac\eta2(K_{S,t}C_t+B_tF_t+C_tG_t)^\top,\\
 T_t^+N_t^\top&=F_t^\top-\frac\eta2(B_t^\top C_t+E_tF_t+F_tG_t)^\top,\\
 T_t^+T_t^\top&=G_t-\frac\eta2(C_t^\top C_t+F_t^\top F_t+G_t^2).
\end{align*}
Let $\rho=\dist_t$.  Since $\|K_{S,t}\|_2\le 2\sigma_1+\rho$, $\rho\le c_{\rm loc}\sigma_r$, $\eta\sigma_1\le c_{\rm step}$, and every other displayed block has norm at most $\rho$, the preceding three identities imply, respectively,
\begin{align*}
 \|T_t^+S_t^\top\|_2
 &\le \rho+\frac\eta2(2\sigma_1\rho+3\rho^2)\le \frac{1}{6}C\rho,\\
 \|T_t^+N_t^\top\|_2
 &\le \rho+\frac{3\eta}{2}\rho^2\le \frac{1}{6}C\rho,\\
 \|T_t^+T_t^\top\|_2
 &\le \rho+\frac{3\eta}{2}\rho^2\le \frac{1}{6}C\rho.
\end{align*}
Consequently,
\[
 \|T_t^+P_{T,t}^\top\|_2
 \le\delta_t(\|T_t^+S_t^\top\|_2+
 \|T_t^+N_t^\top\|_2+\|T_t^+T_t^\top\|_2)
 \le \frac{1}{2}C\delta_t\rho.
\]
Together with \eqref{eq:coefficientbounds}, this proves
\begin{equation}
 \|T_t^+P_{T,t}^\top\|_2+\|P_{T,t}T_t^{+\top}\|_2\le C\delta_t\dist_t,
 \quad
 \|P_{T,t}P_{T,t}^\top\|_2\le C\sigma_1\delta_t^2.
 \label{eq:sharpG}
\end{equation}

The preceding estimates quantify how the empirical gradient perturbs each component of the population dynamics. We now combine the population contraction in \cref{lem:population}, the uniform concentration bound in \cref{lem:concentration}, and the refined null-block estimate \eqref{eq:sixblocks}, \eqref{eq:sharpG} into a single recursion for the block error $\dist_t$. %The resulting lemma establishes the key dichotomy underlying the proof of the main theorem: outside the statistical neighborhood, the error decreases through a quadratic contraction, whereas inside that neighborhood, the iterates remain uniformly stable. It therefore serves as the deterministic bridge between the blockwise dynamics and the final finite-sample convergence guarantee.

\begin{lemma}[Master recursion]
\label{lem:master}
After decreasing $c_{\rm loc}$ and $c_{\rm step}$ in \cref{lem:population} if
necessary, there exist constants $K_{\rm ball},C_{\rm stat}\ge1$ and
$c_{\rm pert},c_{\rm rec}>0$, depending only on $K_A$ and $K_\epsilon$, with
the following property.  Suppose
\begin{equation}
 \dist_t\le c_{\rm loc}\sigma_r,\quad
 0<\eta\le c_{\rm step}/\sigma_1,\quad
 a_n\le c_{\rm pert}/\kappa,\quad
 \kappa b_n\le c_{\rm pert}\sigma_r,
 \label{eq:masterconditions}
\end{equation}
and \eqref{eq:Deltabound} holds.  Define
$\eps=C_{\rm stat}\kappa b_n$.  Then
\begin{align}
 \dist_t\ge2K_{\rm ball}\eps&\quad\Longrightarrow\quad
 \dist_{t+1}\le\dist_t-c_{\rm rec}\eta\dist_t^2,
 \label{eq:outside}\\
 \dist_t\le2K_{\rm ball}\eps&\quad\Longrightarrow\quad
 \dist_{t+1}\le2K_{\rm ball}\eps.
 \label{eq:inside}
\end{align}
\end{lemma}

\begin{proof}
Enlarge a deterministic constant $C_0\ge1$ so that \eqref{eq:Deltabound},
\cref{lem:population}, \eqref{eq:crudecorrections}, and \eqref{eq:sharpG} imply
\begin{align}
 \dist_{t+1}^{\rm nz}
 &\le(1-\mu_0\eta\sigma_r)\dist_t^{\rm nz}
 +C_0\eta\dist_t^2+C_0\eta\sigma_1\delta_t
 +C_0\eta^2\sigma_1\delta_t^2,
 \label{eq:nzmaster}\\
 g_{t+1}
 &\le g_t-\mu_0\eta g_t^2+C_0\eta\delta_t\dist_t
 +C_0\eta^2\sigma_1\delta_t^2,
 \quad g_t:=\|G_t\|_2,
 \label{eq:gmaster}
\end{align}
where $\mu_0=\min\{c_{\rm pop},1\}$ and
$\delta_t\le C_0(a_n\dist_t+b_n)$.  Choose
$0<c_{\rm rec}\le \mu_0/8$.

Decrease $c_{\rm loc},c_{\rm step}$ if necessary such that
\begin{equation}
 \begin{gathered}
 c_{\rm loc}\le1,\quad C_0c_{\rm loc}\le \mu_0/64,\quad
 C_0c_{\rm step}\le \mu_0/64,\\
 c_{\rm rec}c_{\rm loc}c_{\rm step}\le1/4,\quad
 2\mu_0c_{\rm loc}c_{\rm step}\le1,\quad \mu_0c_{\rm step}\le1.
 \end{gathered}
 \label{eq:smallconstants}
\end{equation}
Choose $K_{\rm ball}\ge1$ and then
$0<\theta\le \mu_0/(64C_0)$ so that
\begin{align}
 C_0(c_{\rm loc}+\theta+c_{\rm step}c_{\rm loc}\theta^2)&\le \mu_0/4,\nonumber\\
 C_0(\theta+c_{\rm step}\theta^2)&\le \mu_0/8,\nonumber\\
 C_0\left(\theta+\frac\theta{2K_{\rm ball}}+c_{\rm step}c_{\rm loc}\theta^2\right)&\le \mu_0/2,\nonumber\\
 C_0\left(\frac\theta{2K_{\rm ball}}+\frac{c_{\rm step}\theta^2}{4K_{\rm ball}^2}\right)&\le \mu_0/2.
 \label{eq:absorption}
\end{align}
Next choose $C_{\rm stat}$ large and $c_{\rm pert}$ small so that
\begin{equation}
 \frac{C_0}{2K_{\rm ball}C_{\rm stat}}\le\frac\theta2,\quad
 C_0\left(2K_{\rm ball}c_{\rm pert}+\frac1{C_{\rm stat}}\right)\le\theta,\quad
 2K_{\rm ball}C_{\rm stat}c_{\rm pert}\le\min\{\theta,c_{\rm loc}\}.
 \label{eq:hierarchy}
\end{equation}
%The selection order is noncircular: $(c_{\rm loc},c_{\rm step})\to K_{\rm ball}\to\theta\to C_{\rm stat}\to c_{\rm pert}$.

If $\dist_t=0$, the outside condition can hold only when $\eps=0$; then $\delta_t=0$ and the conclusion follow directly from \eqref{eq:nzmaster} and \eqref{eq:gmaster}. Henceforth, assume that $\dist_t>0$ and $\dist_t\ge2K_{\rm ball}\eps$. Since
$b_n\le\dist_t/(2K_{\rm ball}C_{\rm stat}\kappa)$,
\eqref{eq:hierarchy} gives
\begin{equation}
 \delta_t\le C_0(a_n\dist_t+b_n)\le \frac{C_0}{\kappa}(c_{\rm pert}+1/(2K_{\rm ball}C_{\rm stat}))\dist_t
 \le\frac\theta\kappa\dist_t.
 \label{eq:relative}
\end{equation}
Write $x_t=\dist_t^{\rm nz}$.  In \eqref{eq:nzmaster}, the three positive terms after the contraction are bounded by
\begin{align*}
 C_0\eta\dist_t^2&\le C_0c_{\rm loc}\eta\sigma_r\dist_t,\\
 C_0\eta\sigma_1\delta_t&\le C_0\theta\eta\sigma_r\dist_t,\\
 C_0\eta^2\sigma_1\delta_t^2
 &\le C_0c_{\rm step}c_{\rm loc}\theta^2\eta\sigma_r\dist_t.
\end{align*}
If $x_t\ge\dist_t/2$, the first line of \eqref{eq:absorption} bounds the
sum of the three positive terms by
$(\mu_0/4)\eta\sigma_r\dist_t$.  Hence
\[
 x_{t+1}
 \le \dist_t-\frac{\mu_0}{2}\eta\sigma_r\dist_t
       +\frac{\mu_0}{4}\eta\sigma_r\dist_t
 \le \dist_t-c_{\rm rec}\eta\dist_t^2,
\]
where the last step uses $\dist_t\le c_{\rm loc}\sigma_r\le\sigma_r$ and
$c_{\rm rec}\le \mu_0/8$.  If $x_t<\dist_t/2$, the same forcing bound and
$\mu_0\eta\sigma_r\le \mu_0c_{\rm step}\le1$ give
\[
 x_{t+1}\le\frac12\dist_t+\frac14\dist_t
 =\frac34\dist_t.
\]
Finally, \eqref{eq:smallconstants} gives
$c_{\rm rec}\eta\dist_t\le1/4$, so
$3\dist_t/4\le\dist_t-c_{\rm rec}\eta\dist_t^2$.

For \eqref{eq:gmaster}, if $g_t\ge\dist_t/2$, the negative term is at least
$\mu_0\eta\dist_t^2/4$, whereas \eqref{eq:relative} bounds the two positive terms by
\[
 C_0\theta\eta\dist_t^2/\kappa
 +C_0c_{\rm step}\theta^2\eta\dist_t^2/\kappa^2.
\]
The second line of \eqref{eq:absorption} bounds the sum of these two terms by
$(\mu_0/8)\eta\dist_t^2$.  Therefore, when $g_t\ge\dist_t/2$,
\[
 g_{t+1}\le \dist_t-\frac{\mu_0}{4}\eta\dist_t^2
                   +\frac{\mu_0}{8}\eta\dist_t^2
 \le\dist_t-c_{\rm rec}\eta\dist_t^2.
\]
If $g_t<\dist_t/2$, then
\[
 g_{t+1}\le\frac12\dist_t+\frac{\mu_0}{8}\eta\dist_t^2
 \le\frac34\dist_t
 \le\dist_t-c_{\rm rec}\eta\dist_t^2,
\]
where \eqref{eq:smallconstants} justifies the last two comparisons.  Taking
the maximum of $x_{t+1}$ and $g_{t+1}$ proves \eqref{eq:outside}.

%If $\eps=0$ and $\dist_t=0$, then $b_n=0$ and \eqref{eq:Deltabound} gives $\delta_t=0$. All error blocks are zero, and \eqref{eq:nzmaster}--\eqref{eq:gmaster} yield $\dist_{t+1}=0$. This proves \eqref{eq:inside} for the zero-radius case without division by $\eps$ or $r_{\rm ball}$.

For the remaining case assume $\eps>0$ and let
$\dist_t\le r_{\rm ball}:=2K_{\rm ball}\eps$.  The second and third
conditions in \eqref{eq:hierarchy} give
\begin{equation}
 \delta_t\le\theta\eps/\kappa,\quad
 r_{\rm ball}\le\min\{\theta,c_{\rm loc}\}\sigma_r.
 \label{eq:ballbounds}
\end{equation}
For $x=\dist_t^{\rm nz}\in[0,r_{\rm ball}]$, the three positive terms in
\eqref{eq:nzmaster}, divided by $\eta\sigma_rr_{\rm ball}$, are bounded by
\[
 C_0\theta,\quad \frac{C_0\theta}{2K_{\rm ball}},\quad
 C_0c_{\rm step}c_{\rm loc}\theta^2.
\]
Their sum is at most $\mu_0/2$ by \eqref{eq:absorption}.  Then
\[
 x_{t+1}\le(1-\mu_0\eta\sigma_r)r_{\rm ball}
 +\frac{\mu_0}{2}\eta\sigma_rr_{\rm ball}
 =r_{\rm ball}-\frac{\mu_0}{2}\eta\sigma_rr_{\rm ball}
 \le r_{\rm ball}.
\]

The two positive terms in
\eqref{eq:gmaster}, evaluated at the worst-case radius and divided by
$\eta r_{\rm ball}^2$, are at most $C_0\theta/(2K_{\rm ball})$ and
$C_0c_{\rm step}\theta^2/(4K_{\rm ball}^2)$.  The final line of
\eqref{eq:absorption} shows that their sum is at most $\mu_0/2$. Since
$g\mapsto g-\mu_0\eta g^2$ is increasing on $[0,r_{\rm ball}]$ by \eqref{eq:smallconstants}, evaluation at
$g=r_{\rm ball}$ gives
\[
 g_{t+1}\le r_{\rm ball}-\mu_0\eta r_{\rm ball}^2
 +\frac{\mu_0}{2}\eta r_{\rm ball}^2
 =r_{\rm ball}-\frac{\mu_0}{2}\eta r_{\rm ball}^2\le r_{\rm ball}.
\]
This proves \eqref{eq:inside}.
\end{proof}

%\begin{remark}[Conditional linear phase] \label{rem:linear}
%Equations \eqref{eq:nzmaster} and \eqref{eq:Deltabound} show that a sufficient condition for linear contraction of the nonzero blocks is \begin{equation} \dist_t^{\rm nz}\ge C_{\rm lin}\left( \frac{\dist_t^2}{\sigma_r}+\kappa a_n\dist_t+\kappa b_n\right). \label{eq:linearcondition} \end{equation} Indeed, using $a_n\le1$, $b_n\lesssim\sigma_1$, and $\eta\sigma_1\le c_{\rm step}$, the positive terms in \eqref{eq:nzmaster} are at most \[ C\eta\sigma_r\left( \frac{\dist_t^2}{\sigma_r}+\kappa a_n\dist_t+\kappa b_n\right).\] To see this directly, \eqref{eq:Deltabound} gives $\sigma_1\delta_t\le C\sigma_r\kappa(a_n\dist_t+b_n)$, while $\delta_t\le C(a_n\dist_t+b_n)\le C\sigma_1$ in the local region.  Hence \[ \eta^2\sigma_1\delta_t^2 \le C(\eta\sigma_1)\eta\sigma_1\delta_t \le C\eta\sigma_r\kappa(a_n\dist_t+b_n),\] and the remaining nonlinear term is exactly $\eta\dist_t^2=\eta\sigma_r(\dist_t^2/\sigma_r)$. For sufficiently large $C_{\rm lin}$ this is absorbed into a fixed fraction of $\mu_0\eta\sigma_r\dist_t^{\rm nz}$.  Thus $\dist_{t+1}^{\rm nz}\le(1-c\eta\sigma_r)\dist_t^{\rm nz}$. With $\eta\asymp1/\sigma_1$, the per-step factor is $1-c/\kappa$.  When $G_t$ dominates, the nonzero blocks only track the moving forcing scale $\dist_t^2/\sigma_r$, and the overall rate is governed by the null block. \end{remark}

\section{Proof of the main theorem}
\label{sec:main-proof}

\begin{proof}[Proof of \cref{thm:main}]
Fix the constants in \cref{lem:master}.  Choose
$c_{\rm thm}\le\min\{c_{\rm loc},c_{\rm step}\}$ and then choose
$C_{\rm thm}$ so that it dominates
$C_{\rm samp}$, $c_{\rm pert}^{-2}$, $C_\Delta$, $c_0:=3c_{\rm thm}/\min\{1,c_{\rm thm}c_{\rm rec}\}$, $6K_{\rm ball} C_{\rm stat}$, and the remaining fixed
constants from \cref{lem:master}.  Fix
$c_{\rm prob}\le1$, matching the exponent in \cref{lem:concentration}.
The sample-size condition gives
\[
 a_n=k\sqrt{D/n}\le\frac1{\sqrt{C_{\rm thm}}\kappa}\le\frac{c_{\rm pert}}\kappa, \quad \kappa b_n=\kappa\sigma\sqrt{D/n}
 \le\frac{\sigma_r}{\sqrt{C_{\rm thm}}}\le c_{\rm pert}\sigma_r.
\]
Since $k\ge r\ge1$, $\kappa\ge1$, and the fixed choice satisfies
$C_{\rm thm}\ge C_{\rm samp}$, the same sample-size condition also gives
$n\ge C_{\rm samp}kD$, so \cref{lem:concentration} is applicable.  Thus all
perturbation conditions in \eqref{eq:masterconditions}
hold on its single uniform event.

In the basis $[U_+,U_-,W]$, the six distinct blocks of
$Z_0Z_0^\top-M_+$ are $A_0,B_0,C_0,E_0,F_0,G_0$.  Thus
\begin{equation}
 \dist_0\le\|Z_0Z_0^\top-M_+\|_2
 \le c_{\rm thm}\sigma_r\le c_{\rm loc}\sigma_r.
 \label{eq:initialD}
\end{equation}
Moreover, by \eqref{eq:hierarchy},
$2K_{\rm ball}\eps=2K_{\rm ball}C_{\rm stat}\kappa b_n
\le c_{\rm loc}\sigma_r$.
We show that
$\dist_t\le c_{\rm loc}\sigma_r$ for every $t$, so the local lemma remains valid throughout the trajectory.  Indeed, the claim holds at $t=0$ by
\eqref{eq:initialD}.  If it holds at time $t$, then either
$\dist_t\ge2K_{\rm ball}\eps$, in which case \eqref{eq:outside} gives
$\dist_{t+1}\le\dist_t\le c_{\rm loc}\sigma_r$, or
$\dist_t\le2K_{\rm ball}\eps$, in which case \eqref{eq:inside} and the
preceding bound on $2K_{\rm ball}\eps$ give the same conclusion.  This proves
the induction step.

Let, with the convention $\inf\varnothing=\infty$, denote $ t_\star=\inf\{t\ge0:\dist_t\le2K_{\rm ball}\eps\}.$

Fix any finite $t<t_\star$.  For every $0\le j<t$, the definition of
$t_\star$ gives $\dist_j,\dist_{j+1}>2K_{\rm ball}\eps\ge0$, and
\eqref{eq:outside} gives
\[
 \dist_{j+1}\le\dist_j(1-c_{\rm rec}\eta\dist_j).
\]
The local step-size conditions ensure
$c_{\rm rec}\eta\dist_j\le
c_{\rm rec}c_{\rm step}c_{\rm loc}\sigma_r/\sigma_1\le1/2$.
Hence the factor on the right is positive, and using
$(1-u)^{-1}\ge1+u$ for $u\in[0,1)$, we have 
\[
 \frac1{\dist_{j+1}}
 \ge\frac1{\dist_j(1-c_{\rm rec}\eta\dist_j)}
 \ge\frac1{\dist_j}+c_{\rm rec}\eta.
\]
\eqref{eq:initialD} implies
$1/\dist_0\ge1/(c_{\rm thm}\sigma_r)$, and summing the above inequality from $j=0$ to $t-1$ yields
\begin{equation}
 \dist_t\le\frac1{1/\dist_0+c_{\rm rec}\eta t}
 \le\frac{c_0}{\eta t+1/\sigma_r}\le\frac{C_{\rm thm}}{3(\eta t+1/\sigma_r)}.
 \label{eq:recursionsolution}
\end{equation} If $t_\star<\infty$,
then \eqref{eq:inside}, applied inductively from time $t_\star$, gives
$\dist_t\le2K_{\rm ball}\eps$ for all $t\ge t_\star$.  If
$t_\star=\infty$, the
pre-entry estimate \eqref{eq:recursionsolution} holds for every finite $t$.
Combining these exhaustive cases, we have
\begin{equation}
 \dist_t\le \max\big\{\frac{C_{\rm thm}}{3(\eta t+1/\sigma_r)},2K_{\rm ball} C_{\rm stat}\kappa b_n\big\}\le \frac{1}{3}C_{\rm thm}\max\big\{\frac{1}{\eta t+1/\sigma_r},\kappa b_n\big\}.
 \label{eq:Dfinal}
\end{equation}
Combining \eqref{eq:Dfinal} and
\eqref{eq:productbound}, we prove the maximum-form bound \eqref{eq:mainrate}.  

Assume $\eps_{\rm stat}>0$ and
$t\ge 1/(\eta\eps_{\rm stat})$.  The time condition gives
$1/(\eta t+1/\sigma_r)\le1/(\eta t)\le
\eps_{\rm stat}$, which proves the first statistical inequality.  Moreover,
$\rank(L_tR_t^\top-X^\star)\le k+r\le2k$.  The general singular-value
inequality $\|M\|_F^2\le\rank(M)\|M\|_2^2$ therefore gives
\[
 \|L_tR_t^\top-X^\star\|_F^2
 \le 2k\|L_tR_t^\top-X^\star\|_2^2
 \le 2C_{\rm thm}k\kappa^2\frac Dn\sigma^2,
\]
which proves the second statistical inequality.  Finally, if $\sigma=0$,
then $b_n=\eps=0$ and
\eqref{eq:outside} applies whenever $\dist_t>0$.  If an iterate reaches
$\dist_t=0$, \eqref{eq:inside} with radius zero keeps all subsequent iterates
there; otherwise \eqref{eq:recursionsolution} holds for every finite $t$ and
tends to zero.  In both cases \eqref{eq:productbound} gives exact recovery in
the limit.
\end{proof}

\begin{proof}[Proof of \cref{cor:matched}]
The proof of \cref{thm:main} shows that, 
\begin{equation}
  \dist_t\le c_{\rm loc}\sigma_r
  \quad\text{for every }t\ge0.
  \label{eq:matched-local-bound}
\end{equation}
Set $x_t:=\dist_t^{\rm nz}$ and $g_t:=\|G_t\|_2$. By definition, $\dist_t=\max\{x_t,g_t\}$. If $D=2r$, then $C_t=F_t=G_t=0$. It follows immediately that $\dist_t=x_t$. Suppose now that $k=r$. Then $S_t\in\mathbb R^{r\times r}$ is square.
Recall that $S_tS_t^\top=\Lambda+A_t$. By \eqref{eq:Sbounds} and \eqref{eq:matched-local-bound}, $S_tS_t^\top\succeq (2-c_{\rm loc})\sigma_r I_r$.
Thus $S_t$ is invertible for every $t$, and $$\|S_t^{-1}\|_2^2
  =\frac{1}{\lambda_{\min}(S_tS_t^\top)}
  \le \frac{1}{(2-c_{\rm loc})\sigma_r}.$$ Since $C_t=S_tT_t^\top$, we have $G_t=T_tT_t^\top \notag=C_t^\top S_t^{-\top}S_t^{-1}C_t$
and hence
\begin{align}
  g_t=\|G_t\|_2&\le \|S_t^{-1}\|_2^2\|C_t\|_2^2 
  \le\frac{\|C_t\|_2^2}{(2-c_{\rm loc})\sigma_r}\le\frac{x_t^2}{(2-c_{\rm loc})\sigma_r}.
  \label{eq:matched-G-control}
\end{align}
Moreover, \eqref{eq:matched-local-bound} gives $x_t\le\dist_t\le c_{\rm loc}\sigma_r$ and since $c_{\rm loc}\le1$, we have $\frac{x_t}{(2-c_{\rm loc})\sigma_r} \le \frac{c_{\rm loc}}{2-c_{\rm loc}}\le1$. It follows from \eqref{eq:matched-G-control} that $g_t\le x_t$. Consequently,
\begin{equation}
  \dist_t=\max\{x_t,g_t\}=x_t,
  \quad\text{for every }t\ge0.
  \label{eq:matched-dist-equality}
\end{equation}

We next derive the linear recursion.  On the uniform concentration event,
\eqref{eq:Deltabound} and \eqref{eq:matched-dist-equality} imply
\begin{equation}
  \delta_t\le C_0(a_nx_t+b_n),\quad a_n=k\sqrt{\frac Dn},\quad b_n=\sigma\sqrt{\frac Dn}.
  \label{eq:matched-delta}
\end{equation}
Substituting $\dist_t=x_t$ into \eqref{eq:nzmaster} gives
\begin{align}
  x_{t+1}\le
  (1-\mu_0\eta\sigma_r)x_t
  +C_0\eta x_t^2  +C_0\eta\sigma_1(a_nx_t+b_n) 
  +C_0\eta^2\sigma_1(a_nx_t+b_n)^2.
  \label{eq:matched-expanded-recursion}
\end{align}
We now bound the positive terms on the right-hand side.  First, by
\eqref{eq:matched-local-bound},
\begin{equation}
  C_0\eta x_t^2
  \le
  C_0c_{\rm loc}\eta\sigma_r x_t.
  \label{eq:matched-local-quadratic}
\end{equation}
Next, the sample-size condition in \cref{thm:main} implies $a_n\le\frac{c_{\rm pert}}{\kappa}$ and $\kappa b_n\le c_{\rm pert}\sigma_r$.
Since $\sigma_1=\kappa\sigma_r$, we have
\begin{align}
  C_0\eta\sigma_1a_nx_t
  &=C_0\eta\sigma_r(\kappa a_n)x_t \le
  C_0c_{\rm pert}\eta\sigma_r x_t.
  \label{eq:matched-relative-perturbation}
\end{align}
The part depending only on the observation noise satisfies
\begin{equation} C_0\eta\sigma_1b_n=C_0\eta\sigma_r\kappa b_n.
\label{eq:matched-linear-noise}
\end{equation}
For the squared perturbation term, use
$(a_nx_t+b_n)^2\le2a_n^2x_t^2+2b_n^2$. The first component obeys
\begin{align}
  \eta^2\sigma_1a_n^2x_t^2=
  \left[
    (\eta\sigma_1)a_n^2\frac{x_t}{\sigma_r}
  \right]\eta\sigma_r x_t\le
  c_{\rm step}
  \frac{c_{\rm pert}^2}{\kappa^2}
  c_{\rm loc}\,
  \eta\sigma_r x_t,
  \label{eq:matched-squared-relative}
\end{align}
where we used $\eta\sigma_1\le c_{\rm step}$, $a_n\le\frac{c_{\rm pert}}{\kappa}$  and  $x_t\le c_{\rm loc}\sigma_r$. For the second component, observe that $b_n\le\frac{c_{\rm pert}\sigma_r}{\kappa}$ and $\eta\le\frac{c_{\rm step}}{\sigma_1}
  =\frac{c_{\rm step}}{\kappa\sigma_r}$
and hence $\eta b_n\le\frac{c_{\rm step}c_{\rm pert}}{\kappa^2}$. Consequently,
\begin{align}
  \eta^2\sigma_1b_n^2
  =(\eta b_n)\eta\sigma_r\kappa b_n \le
  \frac{c_{\rm step}c_{\rm pert}}{\kappa^2}
  \eta\sigma_r\kappa b_n.
  \label{eq:matched-squared-noise}
\end{align}

The constants $c_{\rm loc}$, $c_{\rm pert}$, and $c_{\rm step}$ in the proof
of \cref{thm:main} may be chosen sufficiently small, depending only on
$K_A$ and $K_\epsilon$, so that the terms in
\eqref{eq:matched-local-quadratic},
\eqref{eq:matched-relative-perturbation}, and
\eqref{eq:matched-squared-relative} satisfy
\[
  C_0\eta x_t^2
  +C_0\eta\sigma_1a_nx_t
  +2C_0\eta^2\sigma_1a_n^2x_t^2
  \le
  \frac{\mu_0}{2}\eta\sigma_r x_t.
\]
Similarly, \eqref{eq:matched-linear-noise} and
\eqref{eq:matched-squared-noise} give
$$C_0\eta\sigma_1b_n+2C_0\eta^2\sigma_1b_n^2\le C_2\eta\sigma_r\kappa b_n$$ for a constant $C_2>0$ depending only on $K_A$ and $K_\epsilon$. Substitution into \eqref{eq:matched-expanded-recursion} yields
\begin{equation}x_{t+1}\le(1-c\eta\sigma_r)x_t+C_2\eta\sigma_r\kappa b_n,\quad c:=\frac{\mu_0}{2}.
  \label{eq:matched-affine-recursion}
\end{equation}
Denote $q:=1-c\eta\sigma_r<1$. Iterating \eqref{eq:matched-affine-recursion} gives
\begin{align}
x_t\le q^t x_0+ C_2\eta\sigma_r\kappa b_n  \sum_{j=0}^{t-1}q^j \le  q^t x_0+ \frac{C_2\eta\sigma_r\kappa b_n}{1-q} =
(1-c\eta\sigma_r)^t x_0+\frac{C_2}{c}\kappa b_n.
\label{eq:matched-recursion-solution}
\end{align}

By \eqref{eq:productbound} and
\eqref{eq:matched-dist-equality},
$ \|L_tR_t^\top-X^\star\|_2
  \le3\dist_t
  =3x_t$. Combining this inequality with
\eqref{eq:matched-recursion-solution} and $x_0\le\|Z_0Z_0^\top-M_+\|_2$ gives
\begin{align}
  \|L_tR_t^\top-X^\star\|_2
  &\le
  C\left[
    (1-c\eta\sigma_r)^t
    \|Z_0Z_0^\top-M_+\|_2
    +\kappa b_n
  \right] \notag\\
  &\le
  C_{\rm lin}
  \max\left\{
    (1-c_{\rm lin}\eta\sigma_r)^t
    \|Z_0Z_0^\top-M_+\|_2,\,
    \eps_{\rm stat}
  \right\},
  \label{eq:matched-sum-form}
\end{align}
where $\eps_{\rm stat}=\kappa\sigma\sqrt{\frac Dn}=\kappa b_n$, $C_{\rm lin}=2C=6\max\{1,C_2/c\}$ and $c_{\rm lin}=c$. 
If $\sigma=0$, then $b_n=\eps_{\rm stat}=0$, and the preceding estimate
reduces to
\[
  \|L_tR_t^\top-X^\star\|_2
  \le
  C_{\rm lin}
  (1-c_{\rm lin}\eta\sigma_r)^t
  \|Z_0Z_0^\top-M_+\|_2.
\]
Thus the noiseless iterates converge geometrically to exact recovery.
\end{proof}

\paragraph{Numerical validation}
\label{subsec:noisy-numerical}

We show the convergence result of gradient decent for matrix sensing under overparameterized and exactly parameterized factorizations.  The observations follow the Gaussian matrix sensing model $y_i=\langle A_i,V^\star\rangle+\xi_i$, $i=1,\ldots,n$, where the entries of each sensing matrix $A_i$ are independent standard Gaussian random variables and $\xi_i\sim\mathcal{N}(0,\sigma^2)$.  We set $d_1=d_2=50$, $r=3$, $n=1000$ and $\sigma=0.05$. We choose the nonzero singular values of the rank-$3$ target matrix as $(4,3,2)$. Consequently, the condition number is $\kappa=2$.  Gradient descent is run for $1000$ iterations with step size $\eta=0.1$. Both experiments use the rank-$k$ truncated SVD of the isotropic spectral estimator $\frac{1}{n}\sum_{i=1}^{n}y_iA_i$ for initialization. In the overparameterized experiment, the factor width is $k=4>r$. In the exactly parameterized experiment, we set $k=r=3$. We report the product error $\mathcal{E}_t:=\left\|L_tR_t^\top-X^\star\right\|_F$, together with two block error $D_t^{\mathrm{nz}}$ and $D_t^{\mathrm{null}}:=\|G_t\|_2$. The dashed horizontal line represents the statistical reference scale $4\varepsilon_{\mathrm{stat}}=4\kappa\sigma\sqrt{\frac{d_1+d_2}{n}}\approx 1.2649\times10^{-1}$.

\begin{figure}[t]
    \centering
    \includegraphics[width=\textwidth]
    {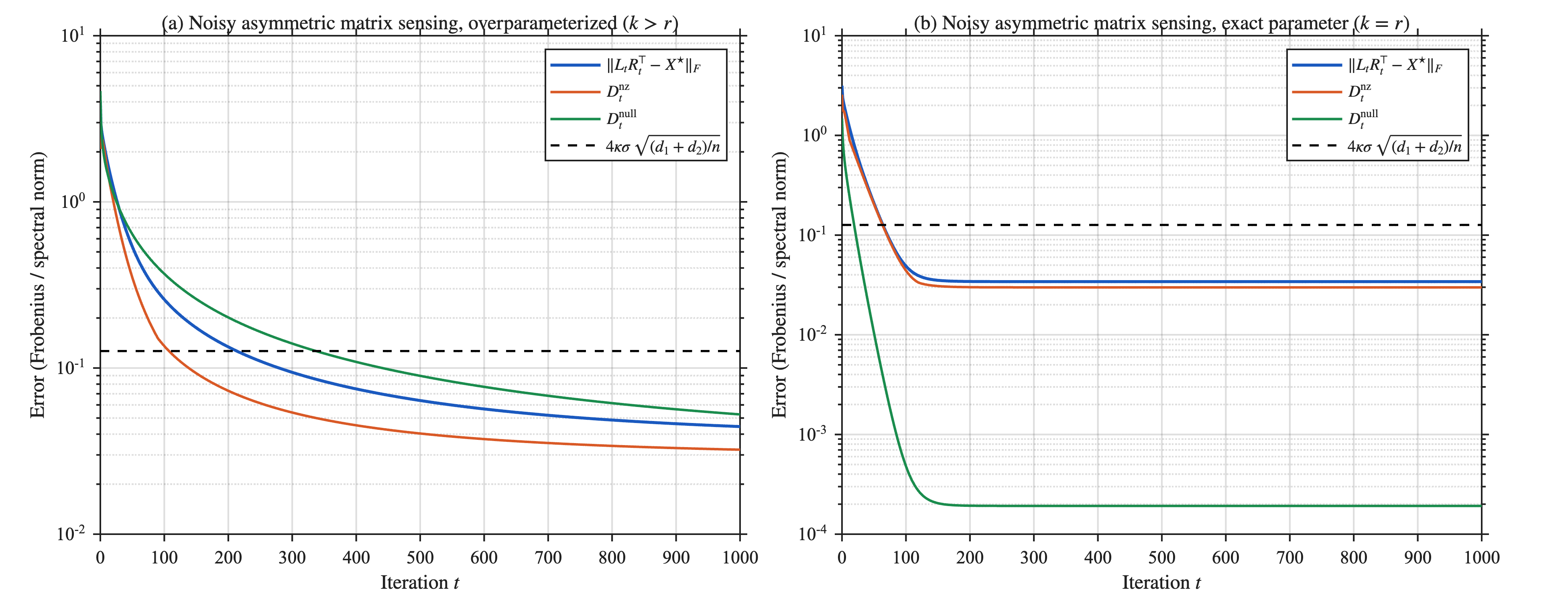}
    \caption{Convergence behavior under noisy asymmetric matrix sensing.
    \textnormal{(a)} Overparameterized factorization with $k=4>r=3$.
    \textnormal{(b)} Exactly parameterized factorization with $k=r=3$.
    The blue, orange, and green curves represent the product error $\|L_tR_t^\top-X^\star\|_F$, the block error $D_t^{\mathrm{nz}}$, and $D_t^{\mathrm{null}}:=\|G_t\|_2$, respectively.  The black dashed line denotes $4\varepsilon_{\mathrm{stat}}=4\kappa\sigma\sqrt{(d_1+d_2)/n}$.}
    \label{fig:noisy-two-case}
\end{figure}

As shown in Fig.~\ref{fig:noisy-two-case}(a), in the overparameterized case, $D_t^{\mathrm{nz}}$ converges linearly to the statistical error $4\varepsilon_{\mathrm{stat}}$, while $D_t^{\mathrm{null}}$ and $\mathcal{E}_t$ converge sublinearly to the statistical error. As shown in Fig.~\ref{fig:noisy-two-case}(b), in the exact-parameter case, $D_t^{\mathrm{nz}}$, $D_t^{\mathrm{null}}$, and $\mathcal{E}_t$ converge linearly to the statistical error. This verifies the conclusion of our theorem.

\section{Conclusion}
\label{sec:limitations}

We analyzed factorized gradient descent for overparameterized asymmetric matrix sensing with a balancing regularizer. Under sub-Gaussian measurements, noisy observations, and a local initialization, we established
\[
\|L_tR_t^\top-X^\star\|_2
\lesssim
\max\left\{
\frac{1}{\eta t+1/\sigma_r},
\,
\kappa\sigma\sqrt{\frac{d_1+d_2}{n}}
\right\}.
\]
Hence, the iterates converge at an \(O(1/t)\) rate to the statistical error scale. The spectral decomposition of the lifted dynamics identifies the source of this rate: the positive and negative spectral blocks have linear restoring terms, whereas the redundant null block contracts only quadratically and therefore governs the worst-case \(O(1/t)\) behavior. A population example shows that this rate is tight over the class of local initializations considered here.

\appendix

\section{Complete second-order population remainders}
\label{app:remainders}

This appendix expands every second-order remainder in \eqref{eq:remainders}.  From
\begin{equation}
 [\mathcal R_t\mathcal G_t\mathcal R_t]_{ij}
 =
 \sum_{a=1}^3\sum_{b=1}^3
 [\mathcal R_t]_{ia}
 [\mathcal G_t]_{ab}
 [\mathcal R_t]_{bj},
 \label{eq:indexpaths}
\end{equation}
the \((1,1)\) block gives
\begin{equation}
\begin{aligned}
 R_{A,t}
 =\frac{\eta^2}{4}\bigl(
 &A_t\Lambda A_t+A_t^3
 +A_tB_tB_t^\top+A_tC_tC_t^\top
 +B_tB_t^\top A_t+B_tE_tB_t^\top+B_tF_tC_t^\top\\
 &+C_tC_t^\top A_t+C_tF_t^\top B_t^\top+C_tG_tC_t^\top
 \bigr).
\end{aligned}
\label{eq:RAfull}
\end{equation}
The \((1,2)\) block gives
\begin{equation}
\begin{aligned}
 R_{B,t}
 =\frac{\eta^2}{4}\bigl(
 &A_t\Lambda B_t+A_t^2B_t+A_tB_tE_t+A_tB_t\Lambda
 +A_tC_tF_t^\top\\
 &+B_tB_t^\top B_t+B_tE_t^2+B_tE_t\Lambda
 +B_tF_tF_t^\top+C_tC_t^\top B_t\\
 &+C_tF_t^\top E_t+C_tF_t^\top\Lambda+C_tG_tF_t^\top
 \bigr).
\end{aligned}
\label{eq:RBfull}
\end{equation}
The \((1,3)\) block gives
\begin{equation}
\begin{aligned}
 R_{C,t}
 =\frac{\eta^2}{4}\bigl(
 &A_t\Lambda C_t+A_t^2C_t+A_tB_tF_t+A_tC_tG_t
 +B_tB_t^\top C_t+B_tE_tF_t+B_tF_tG_t\\
 &+C_tC_t^\top C_t+C_tF_t^\top F_t+C_tG_t^2
 \bigr).
\end{aligned}
\label{eq:RCfull}
\end{equation}
For the \((2,2)\) block,
\[
 (E_t+\Lambda)E_t(E_t+\Lambda)-\Lambda E_t\Lambda
 =
 E_t^3+E_t^2\Lambda+\Lambda E_t^2,
\]
thus
\begin{equation}
\begin{aligned}
 R_{E,t}
 =\frac{\eta^2}{4}\bigl(
 &B_t^\top\Lambda B_t+B_t^\top A_tB_t
 +B_t^\top B_tE_t+B_t^\top B_t\Lambda
 +B_t^\top C_tF_t^\top\\
 &+E_tB_t^\top B_t+\Lambda B_t^\top B_t
 +E_t^3+E_t^2\Lambda+\Lambda E_t^2
 +E_tF_tF_t^\top+\Lambda F_tF_t^\top\\
 &+F_tC_t^\top B_t+F_tF_t^\top E_t
 +F_tF_t^\top\Lambda+F_tG_tF_t^\top
 \bigr).
\end{aligned}
\label{eq:REfull}
\end{equation}
Finally, the \((2,3)\) block gives
\begin{equation}
\begin{aligned}
 R_{F,t}
 =\frac{\eta^2}{4}\bigl(
 &B_t^\top\Lambda C_t+B_t^\top A_tC_t
 +B_t^\top B_tF_t+B_t^\top C_tG_t
 +E_tB_t^\top C_t+\Lambda B_t^\top C_t\\
 &+E_t^2F_t+\Lambda E_tF_t
 +E_tF_tG_t+\Lambda F_tG_t
 +F_tC_t^\top C_t+F_tF_t^\top F_t+F_tG_t^2
 \bigr).
\end{aligned}
\label{eq:RFfull}
\end{equation}

Denote $M_t=\{A_t,B_t,C_t,E_t,F_t,G_t\}$.
Let $\rho=\dist_t$. Since $\|\Lambda\|_2=2\sigma_1$, for arbitrary blocks
\(H_{1,t},H_{2,t},H_{3,t}\) from $M_t$, we have 
\[
 \|H_{1,t}\Lambda H_{2,t}\|_2 ,\ \|\Lambda H_{1,t}H_{2,t}\|_2,\ 
 \|H_{1,t}H_{2,t}\Lambda\|_2
 \le 2\sigma_1\rho^2,
\]
and
\[
 \|H_{1,t}H_{2,t}H_{3,t}\|_2
 \le \rho^3
 \le c_{\rm loc}\sigma_r\rho^2
 \le c_{\rm loc}\sigma_1\rho^2.
\]
These four patterns cover, respectively, terms such as
\(A_t\Lambda A_t\), \(\Lambda E_tF_t\), \(B_tE_t\Lambda\) and
\(B_tF_tC_t^\top\). Inspection of
\eqref{eq:RAfull}--\eqref{eq:RFfull} shows that every displayed atom
is one of these types. Thus, there exist constant $C$, such that 
\[
 \max_{H\in\{A,B,C,E,F\}}\|R_{H,t}\|_2
 \le C\eta^2\sigma_1\rho^2.
\]
Here, $C$ can be chosen as 23/4.
Finally, \(\eta\sigma_1\le c_{\rm step}\) implies $\eta^2\sigma_1\rho^2 \le c_{\rm step}\eta\rho^2$, proving \eqref{eq:remainderbound}.

\bibliographystyle{plain}
{\small
\setlength{\parskip}{0pt}
\bibliography{references}

@inproceedings{bi2022,
  title={Local and global linear convergence of general low-rank matrix recovery problems},
  author={Bi, Yingjie and Zhang, Haixiang and Lavaei, Javad},
  booktitle={Proceedings of the AAAI Conference on Artificial Intelligence},
  volume={36},
  number={9},
  pages={10129--10137},
  year={2022}
}

@article{candesplan2011,
  author = {E. J. Cand\`es and Y. Plan},
  title = {Tight oracle inequalities for low-rank matrix recovery from a minimal number of noisy random measurements},
  year = {2011},
  journal = {IEEE Transactions on Information Theory},
  volume = {57},
  pages = {2342--2359},
  number = {4},
}

@misc{chenwainwright2015,
  author = {Y. Chen and M. J. Wainwright},
  title = {Fast low-rank estimation by projected gradient descent: General statistical and algorithmic guarantees},
  year = {2015},
  howpublished = {arXiv:1509.03025},
}

@article{chi2019,
  author = {Y. Chi and Y. M. Lu and Y. Chen},
  title = {Nonconvex optimization meets low-rank matrix factorization: An overview},
  year = {2019},
  journal = {IEEE Transactions on Signal Processing},
  volume = {67},
  pages = {5239--5269},
  number = {20},
}

@article{dirksen2016,
  author = {Sjoerd Dirksen},
  title = {Dimensionality Reduction with Subgaussian Matrices: A Unified Theory},
  year = {2016},
  journal = {Foundations of Computational Mathematics},
  volume = {16},
  pages = {1367--1396},
  number = {5},
}

@inproceedings{ge2017,
  author = {Rong Ge and Chi Jin and Yi Zheng},
  title = {No Spurious Local Minima in Nonconvex Low Rank Problems: A Unified Geometric Analysis},
  year = {2017},
  booktitle = {Proceedings of the 34th International Conference on Machine Learning},
  pages = {1233--1242},
}

@inproceedings{gunasekar2017,
  author = {S. Gunasekar and B. E. Woodworth and S. Bhojanapalli and B. Neyshabur and N. Srebro},
  title = {Implicit regularization in matrix factorization},
  year = {2017},
  booktitle = {Advances in Neural Information Processing Systems},
}

@inproceedings{jin2023,
  author = {J. Jin and Z. Li and K. Lyu and S. S. Du and J. D. Lee},
  title = {Understanding incremental learning of gradient descent: A fine-grained analysis of matrix sensing},
  year = {2023},
  booktitle = {Proceedings of the 40th International Conference on Machine Learning},
  pages = {15200--15238},
}

@inproceedings{li2018,
  author = {Y. Li and T. Ma and H. Zhang},
  title = {Algorithmic regularization in over-parameterized matrix sensing and neural networks with quadratic activations},
  year = {2018},
  booktitle = {Proceedings of the 31st Conference on Learning Theory},
  pages = {2--47},
}

@misc{liu2025apgd,
  author = {Zhiyu Liu and Zhi Han and Yandong Tang and Hai Zhang and Shaojie Tang and Yao Wang},
  title = {Efficient Over-parameterized Matrix Sensing from Noisy Measurements via Alternating Preconditioned Gradient Descent},
  year = {2025},
  howpublished = {arXiv:2502.00463v3},
}

@article{mali2021,
  author = {Cong Ma and Yuanxin Li and Yuejie Chi},
  title = {Beyond Procrustes: Balancing-Free Gradient Descent for Asymmetric Low-Rank Matrix Sensing},
  year = {2021},
  journal = {IEEE Transactions on Signal Processing},
  volume = {69},
  pages = {867--877},
}

@article{park2018,
  author = {Dohyung Park and Anastasios Kyrillidis and Constantine Caramanis and Sujay Sanghavi},
  title = {Finding Low-Rank Solutions via Non-Convex Matrix Factorization, Efficiently and Provably},
  year = {2018},
  journal = {SIAM Journal on Imaging Sciences},
  volume = {11},
  pages = {2165--2204},
  number = {4},
}

@book{rao1991,
  author = {M. M. Rao and Z. D. Ren},
  title = {Theory of Orlicz Spaces},
  year = {1991},
  publisher = {Marcel Dekker, New York},
}

@article{recht2010,
  author = {B. Recht and M. Fazel and P. A. Parrilo},
  title = {Guaranteed minimum-rank solutions of linear matrix equations via nuclear norm minimization},
  year = {2010},
  journal = {SIAM Review},
  volume = {52},
  pages = {471--501},
  number = {3},
}

@article{tong2021,
  author = {Tian Tong and Cong Ma and Yuejie Chi},
  title = {Accelerating Ill-Conditioned Low-Rank Matrix Estimation via Scaled Gradient Descent},
  year = {2021},
  journal = {Journal of Machine Learning Research},
}

@inproceedings{tu2016,
  author = {Stephen Tu and Ross Boczar and Max Simchowitz and Mahdi Soltanolkotabi and Benjamin Recht},
  title = {Low-rank Solutions of Linear Matrix Equations via Procrustes Flow},
  year = {2016},
  booktitle = {Proceedings of the 33rd International Conference on Machine Learning},
  pages = {964--973},
}

@book{vershynin2018,
  author = {R. Vershynin},
  title = {High-Dimensional Probability: An Introduction with Applications in Data Science},
  year = {2018},
  publisher = {Cambridge University Press},
}

@inproceedings{wen2012,
  author = {Zhiyuan Weng and Xin Wang},
  title = {Low-rank matrix completion for array signal processing},
  booktitle = {IEEE International Conference on Acoustics, Speech and Signal Processing (ICASSP)},
  pages = {2697--2700},
  year = {2012},
}

@article{wind2025,
    author = {Wind, Johan Sokrates},
    title = {Implicit regularization makes overparameterized asymmetric matrix sensing robust to perturbations},
    journal = {Philosophical Transactions of the Royal Society A: Mathematical, Physical and Engineering Sciences},
    volume = {384},
    number = {2324},
    pages = {20250120},
    year = {2026},
    month = {07},
    }

@inproceedings{xiong2024,
  author = {Nuoya Xiong and Lijun Ding and Simon S. Du},
  title = {How Over-Parameterization Slows Down Gradient Descent in Matrix Sensing: The Curses of Symmetry and Initialization},
  year = {2024},
  booktitle = {International Conference on Learning Representations},
}

@inproceedings{xu2023,
  author = {Xingyu Xu and Yandi Shen and Yuejie Chi and Cong Ma},
  title = {The Power of Preconditioning in Overparameterized Low-Rank Matrix Sensing},
  year = {2023},
  booktitle = {Proceedings of the 40th International Conference on Machine Learning},
  pages = {38611--38654},
}

@inproceedings{zhang2021,
  author = {Jialun Zhang and Salar Fattahi and Richard Y. Zhang},
  title = {Preconditioned Gradient Descent for Over-Parameterized Nonconvex Matrix Factorization},
  year = {2021},
  booktitle = {Advances in Neural Information Processing Systems},
}

@inproceedings{zhang2024noisy,
  author = {Gavin Zhang and Hong-Ming Chiu and Richard Y. Zhang},
  title = {Fast and Accurate Estimation of Low-Rank Matrices from Noisy Measurements via Preconditioned Non-Convex Gradient Descent},
  year = {2024},
  booktitle = {Proceedings of the 27th International Conference on Artificial Intelligence and Statistics},
  pages = {3772--3780},
}

@inproceedings{zhao2010,
  author = {Bo Zhao and Justin P. Haldar and Cornelius Brinegar and Zhi-Pei Liang},
  title = {Low rank matrix recovery for real-time cardiac MRI},
  booktitle = {IEEE International Symposium on Biomedical Imaging: From Nano to Macro},
  pages = {996--999},
  year = {2010},
}

@article{zhu_geometry2021,
  title={The global optimization geometry of low-rank matrix optimization},
  author={Zhu, Zhihui and Li, Qiuwei and Tang, Gongguo and Wakin, Michael B},
  journal={IEEE Transactions on Information Theory},
  volume={67},
  number={2},
  pages={1308--1331},
  year={2021},
  publisher={IEEE}
}

@article{zhuo2024,
  author = {Jiacheng Zhuo and Jeongyeol Kwon and Nhat Ho and Constantine Caramanis},
  title = {On the Computational and Statistical Complexity of Over-parameterized Matrix Sensing},
  year = {2024},
  journal = {Journal of Machine Learning Research},
  volume = {25},
  pages = {1--47},
  number = {169},
}

@article{gross2010quantum,
  author  = {David Gross and Yi-Kai Liu and Steven T. Flammia
             and Stephen Becker and Jens Eisert},
  title   = {Quantum State Tomography via Compressed Sensing},
  journal = {Physical Review Letters},
  volume  = {105},
  number  = {15},
  pages   = {150401},
  year    = {2010},
  doi     = {10.1103/PhysRevLett.105.150401}
}

@article{dokmanic2015edm,
  author  = {Ivan Dokmani{\'c} and Reza Parhizkar and Juri Ranieri
             and Martin Vetterli},
  title   = {Euclidean Distance Matrices: Essential Theory,
             Algorithms, and Applications},
  journal = {IEEE Signal Processing Magazine},
  volume  = {32},
  number  = {6},
  pages   = {12--30},
  year    = {2015},
  doi     = {10.1109/MSP.2015.2398954}
}

@inproceedings{arora2019deep,
  author    = {Sanjeev Arora and Nadav Cohen and Wei Hu and Yuping Luo},
  title     = {Implicit Regularization in Deep Matrix Factorization},
  booktitle = {Advances in Neural Information Processing Systems},
  volume    = {32},
  year      = {2019}
}

@inproceedings{li2021greedy,
  author    = {Zhiyuan Li and Yuping Luo and Kaifeng Lyu},
  title     = {Towards Resolving the Implicit Bias of Gradient Descent
               for Matrix Factorization: Greedy Low-Rank Learning},
  booktitle = {International Conference on Learning Representations},
  year      = {2021}
}

@inproceedings{zheng2015,
  author    = {Qinqing Zheng and John Lafferty},
  title     = {A Convergent Gradient Descent Algorithm for Rank
               Minimization and Semidefinite Programming from
               Random Linear Measurements},
  booktitle = {Advances in Neural Information Processing Systems},
  volume    = {28},
  year      = {2015},
}

@inproceedings{stoger2021,
  author    = {Dominik St{\"o}ger and Mahdi Soltanolkotabi},
  title     = {Small Random Initialization Is Akin to Spectral
               Learning: Optimization and Generalization Guarantees
               for Overparameterized Low-Rank Matrix Reconstruction},
  booktitle = {Advances in Neural Information Processing Systems},
  volume    = {34},
  pages     = {23831--23843},
  year      = {2021},
}

@article{zhu2018,
  author  = {Zhihui Zhu and Qiuwei Li and Gongguo Tang
             and Michael B. Wakin},
  title   = {Global Optimality in Low-Rank Matrix Optimization},
  journal = {IEEE Transactions on Signal Processing},
  volume  = {66},
  number  = {13},
  pages   = {3614--3628},
  year    = {2018},
  doi     = {10.1109/TSP.2018.2835403}
}

@article{li2019,
  author  = {Xingguo Li and Junwei Lu and Raman Arora and Jarvis Haupt
             and Han Liu and Zhaoran Wang and Tuo Zhao},
  title   = {Symmetry, Saddle Points, and Global Optimization
             Landscape of Nonconvex Matrix Factorization},
  journal = {IEEE Transactions on Information Theory},
  volume  = {65},
  number  = {6},
  pages   = {3489--3514},
  year    = {2019},
  doi     = {10.1109/TIT.2019.2898663}
}

@inproceedings{zhang2021a,
  author    = {Haixiang Zhang and Yingjie Bi and Javad Lavaei},
  title     = {General Low-Rank Matrix Optimization:
               Geometric Analysis and Sharper Bounds},
  booktitle = {Advances in Neural Information Processing Systems},
  volume    = {34},
  pages     = {27369--27380},
  year      = {2021}
}

@inproceedings{fang2021modeling,
  title={Modeling from features: a mean-field framework for over-parameterized deep neural networks},
  author={Fang, Cong and Lee, Jason and Yang, Pengkun and Zhang, Tong},
  booktitle={Conference on learning theory},
  pages={1887--1936},
  year={2021},
  organization={PMLR}
}

@article{jacot2018neural,
  title={Neural tangent kernel: Convergence and generalization in neural networks},
  author={Jacot, Arthur and Gabriel, Franck and Hongler, Cl{\'e}ment},
  journal={Advances in neural information processing systems},
  volume={31},
  year={2018}
}

@article{zou2020gradient,
  title={Gradient descent optimizes over-parameterized deep ReLU networks},
  author={Zou, Difan and Cao, Yuan and Zhou, Dongruo and Gu, Quanquan},
  journal={Machine learning},
  volume={109},
  number={3},
  pages={467--492},
  year={2020},
  publisher={Springer}
}
}

\end{document}